\documentclass[12pt]{article}
\usepackage{amsmath}
\usepackage{amssymb}
\usepackage{theorem}
\usepackage{enumerate}
\usepackage{hyperref}

\numberwithin{equation}{section}

\newtheorem{thm}{Theorem}[section]
\newtheorem{lemma}[thm]{Lemma}

\newtheorem{prop}[thm]{Proposition}
\newtheorem{cor}[thm]{Corollary}

{\theorembodyfont{\rmfamily}
\newtheorem{defn}[thm]{Definition}

\newtheorem{rmk}[thm]{Remark}

\newtheorem{notation}[thm]{Notation}
}
\newenvironment{proofof}[1]{\noindent {\bf Proof of #1.}}{ \hfill $\Box$\\ }

\newcommand{\hj}{{\underline j}}
\newcommand{\hk}{{\underline k}}
\newcommand{\hl}{{\underline l}}

\newcommand{\cB}{{\mathbb B}}

\newcommand{\cF}{{\mathcal F}}

\newcommand{\tD}{{\widetilde D}}

\newcommand{\tM}{{\widetilde M}}

\newcommand{\tW}{{\widetilde W}}
\newcommand{\tZ}{{\widetilde Z}}
\newcommand{\tcS}{\widetilde{\mathcal{S}}}

\newcommand{\bd}{\overline{d}}
\newcommand{\bW}{\overline{W}}

\newcommand{\qed}{\hfill \mbox{\raggedright \rule{.07in}{.1in}}}

\newenvironment{proof}{\vspace{1ex}\noindent{\bf
Proof}\hspace{0.5em}}{\hfill\qed\vspace{1ex}}
\newenvironment{pfof}[1]{\vspace{1ex}\noindent{\bf Proof of
#1}\hspace{0.5em}}{\hfill\qed\vspace{1ex}}

\def\R{\mathbb{R}}
\def\N{\mathbb{N}}
\def\Z{\mathbb{Z}}
\def\E{\mathbb{E}}

\def\pP{\mathbb{P}}
\def\T{\mathbb{T}}
\def\cB{\mathcal{B}}

\def\cN{\mathcal{N}}
\def\cS{\mathcal{S}}

\newcommand{\bbS}{{{\mathbb S}^{d-1}}}

\newcommand{\eps}{{\varepsilon}}

\begin{document}

\title{Generalized law of the iterated logarithm for the Lorentz gas with infinite horizon}

\author{P\'eter B\'alint\thanks{
 Department of Stochastics, Institute of Mathematics,
Budapest University of Technology and Economics,
M{\H u}egyetem rkp. 3., H-1111, Budapest, Hungary and HUN-REN--BME Stochastics Research Group,
Budapest University of Technology and Economics,
M{\H u}egyetem rkp. 3., H-1111 Budapest, Hungary and  Erd\H{o}s Center of the HUN-REN Alfr\'ed R\'enyi Institute of Mathematics, Re\'altanoda utca 13-15, H-1053, Budapest, Hungary;
\textit{balint.peter@ttk.bme.hu}. Support of the NKFIH Research Fund, grants 142169 and 144059, is thankfully acknowledged.}\quad and\
Dalia Terhesiu
\thanks{Institute of Mathematics, University of Leiden,
	Niels Bohrweg 1, 2333 CA Leiden, The Netherlands,
		{\it daliaterhesiu@gmail.com}}}

\date{\today}
\maketitle

\begin{abstract}
 We obtain a generalized law of the iterated logarithm for a class of dependent processes with superdiffusive behaviour. Our results apply  to the Lorentz gas with infinite horizon.
\end{abstract}

 \section{Introduction and main results}

In this paper we study almost sure limit phenomena for a class of weakly dependent processes of variables with heavy tailed distributions. In particular, we are interested in obtaining results analogous to
the law of the iterated logarithm and the almost sure invariance principle, but instead of the standard diffusive case, for the non-standard domain of attraction of the normal law. Even in the i.i.d. setting this is a
non-trivial problem, which was investigated by several authors, most importantly by Einmahl in \cite{Ein07},\cite{Ein09}, see below for further details. Our goal is to extend the results of Einmahl
to a class of dependent processes.

\subsection{The infinite horizon Lorentz gas\label{ss:lgintro}}

We hope that the approach presented in this paper can be later applied to several situations, yet, our primary motivation is a particulary important model
of mathematical physics: the periodic Lorentz gas with infinite horizon. Let us give a brief description of this model; for further details, see the references, in particular \cite{ChDo09}, \cite{PT20}, \cite{SV07}.
The $\Z^d$-periodic Lorentz gas describes the evolution of a point particle, equipped with unit speed, moving in the $\Z^d$-periodic domain $\tD$ contained either in the plane $\R^2$ (if $d=2$)
or in the tube $\R\times\T$ (if $d=1$). Within $\tD$ the motion of the particle is uniform and the collisions at the boundary $\partial\tD$ are elastic (equality of pre-collision and post-collision angles with the normal vector of the boundary at the point of impact).
The Lorentz gas map $\hat T:\hat M\to\hat M$ is the collision map on the two-dimensional phase space (position in $\partial\tD$ and post-collision velocity, parametrized by the angle made with the normal vector) given by $\hat M=\partial\tD\times (-\pi/2,\pi/2)$.

We assume periodicity, that is, $\tD$ is the lifted domain (by the canonical projection from $\R^2$ (if $d=2$) or from $\R\times\T$ (if $d=1$) onto $\T^2$ )
of
$D=\T^2\setminus Q$, where $Q\subset\T^2$ is a finite union of convex obstacles (scatterers)
with $C^3$ boundaries and nonvanishing curvature, and pairwise disjoint boundaries. Here $\T^2$ denotes the two dimensional torus.
The Sinai (dispersing) billiard map $T:M\to M$ corresponding to the associated collision map is obtained from $\hat T:\hat M\rightarrow\hat M$ by quotienting.
We denote by $\mu$ the unique ergodic
$T$-invariant smooth (absolutely continuous) probability measure on $M$.

The Lorentz gas map $\hat T:\hat M\to\hat M$ can be viewed as a $\Z^d$-cover, $d=1,2$, of the dispersing billiard $(T,M,\mu)$
by the
cell-change function
 $\kappa:M\to\Z^d$.
We distinguish two main cases: when $\kappa$ is bounded, the horizon is finite, and when $\kappa$ is unbounded, the horizon is \emph{infinite}.
Similarly, one may consider the flight function $\Phi:M\to\R^2$ given by the difference in $\R^2$ between consecutive collision points. The following coboundary relation is well-known (see eg.~\cite{SV07}) : $\Phi$ can be written as $\Phi=\kappa+\chi\circ T-\chi$,
where $\chi$ is bounded. Hence finite/infinite horizon can be also characterized by boundedness/unboundedness of $\Phi$. The Birkhoff sums $S_n\kappa= \sum_{i=0}^{n-1} \kappa\circ T^i$ and
$S_n\Phi= \sum_{i=0}^{n-1} \Phi\circ T^i$ describe the position of the particle after $n$ collisions, on $\Z^d$ and on $\R^2$, respectively. When choosing the initial point according to $\mu$, these quanitites can be regarded sums of (weakly) dependent, vector valued random variables.

The finite and infinite horizon Lorentz gases are popular mechanical models of diffusion and superdiffusion, respectively. It is a classical result (\cite{BuSi}) that for finite horizon $S_n\kappa$ (and thus $S_n\Phi$) follows a central limit theorem (CLT) with the standard diffusive scaling $\sqrt{n}$. On the other hand, for infinite horizon,
under certain nondegeneracy conditions,~\cite{SV07} proved that $S_n\kappa$ satisfies a nonstandard CLT (along with a local limit theorem) with positive-definite covariance matrix $\Sigma\in \R^{d\times d}$ and superdiffusive
normalisation $a_n=\sqrt{C n\log n}$ (for some constant $C>0$), that is
\[
\frac{S_n\kappa}{a_n} \text{ converges in distribution to } \mathcal{N}(0,\Sigma^2).
\]
These nondegeneracy conditions are as follows: a) in the case $d=2$,
one requires that there exist at least two nonparallel collisionless trajectories
in the interior of $\tD$; b) For $d=1$, one requires
that there exists a collisionless trajectory not orthogonal to the direction of the $\Z$-cover, which is equivalent to our assumption that $\kappa:M\to\Z$
is unbounded.

We recall that a different proof (via standard pairs) of the above mentioned non standard CLT has been provided
in~\cite{ChDo09} along with the corresponding weak invariance principle (WIP).
For other recent results on the Lorentz gas map and flow with infinite horizon, see~\cite{BBT, PT20, MPT, PT23}.

\begin{notation}\label{n:LLL}
Throughout the paper we will use the notation $\ell(t)$ for various slowly varying functions (see eg.~\cite{Rvaceva54}). Also, we will use the standard notations in probability,
$L(t)=\max(1,\log(t))$; $LL(t)=L(L(t))$, $LLL(t)=L(LL(t))$ etc.
\end{notation}

\subsection{Classical ASIP and LIL}

 The almost sure invariance principle (ASIP), also known as the strong invariance principle (SIP) or Strassen's invariance principle, is a much studied problem in probability theory, see \cite{Zaitsurv} for a survey.
 Let $S_n=\sum_{i=1}^n Y_i $ denote the sum of a stationary sequence of random variables taking values in $\R^d$, $d\ge 1$; then $S_n$ satisfies an ASIP
 if there exists some
 non-degenerate covariance matrix $\Sigma^2$ such that, possibly on an enlarged probability space,
 \begin{equation}\label{eq:standardasip}
 |S_n - \Sigma \cdot W(n)| = o(n^{1/2-\eps}), \text{ for some }\eps>0.
 \end{equation}
 almost surely as $n\to\infty$, where $W(n)=\sum_{i=1}^n Z_i $ denotes standard Brownian motion on $\R^d$ at time $n$, or equivalently, the variables $Z_i$ are i.i.d. standard Gaussians. It is known that~\eqref{eq:standardasip}
 implies the law of the iterated logarithms (LIL) (by knowing the LIL for the Brownian motion):
 \begin{equation}\label{eq:standardlil}
\limsup_{n\to\infty}  \dfrac{|S_n|}{\sqrt{n \cdot L(Ln)}} =a \qquad \text{ almost surely, for some } a\in (0,\infty).
 \end{equation}

 It can be assumed that the stationary sequence arises as $Y_i=Y\circ T^i$ for some probability preserving transformation (or dynamical system) $T:\Omega\to \Omega$ and $Y:\Omega \to\R^d$.

%

 The problem of the ASIP (with rates) has been studied for various classes of dynamical
 systems and observables -- that satisfy the classical central limit theorem (CLT) and weak invariance principle (WIP) -- by several authors, via different methods, see~\cite{merl, mNicol,mNicol2, Gou-asip,Kor18, CMKD} and
 references therein.
A common feature of these works is that they assume (along with some weak dependence  conditions) that the random variables satisfy $Y\in L^p$ for some $p>2$, and prove the ASIP with some rate
that depends on $p$.

\subsection{A certain ASIP and generalized LIL for independent sequences }

 A certain generalized form of the ASIP and of the LIL for i.i.d.\ random variables   taking values in $\R^d$ with infinite variance but still in the domain of attraction of a normal law  were obtained in the works of Einmahl~\cite{Ein87, Ein88, Ein07,EinLi}.
We will focus on the case of $\R^d$  regularly varying random variables $(Y_j)_{ j\ge 0}$ of index $-2$, but remark that in the case $d=1$, the results in~\cite{Ein07}
 are not limited to regular variation.
For the case $d=1$, regular variation of index $-2$  means that, as $t \to \infty$,
\[
 \pP(Y > t) = (p + o(1))\ell(t)t^{-2},\quad \pP(Y < -t) = (q + o(1))\ell(t)t^{-2}
\]
for some $p,q\ge 0, p+q=1$ and some slowly varying function $\ell$.

To recall the more general definition of regular variation in $\R^d$ (see~\cite{Rvaceva54}) let $\bbS = \{t \in \R^d : |t| = 1 \}$ denote the unit sphere in $\R^d$.
(Throughout, $|\;|$ denotes the Euclidean norm.)

\begin{defn}
\label{def-reg}
An $\R^d$-valued random variable $Y$ is
    \emph{regularly varying} with index $-\alpha<0$ if
    there exists a Borel probability measure $\sigma$
    on $\bbS$, such that
    \[
        \lim_{t \to \infty}
        \frac{\pP(|Y| > \lambda t, \ Y / |Y| \in A)}{\pP(|Y| > t)}
        = \lambda^{-\alpha} \sigma(A)
    \]
    for all $\lambda > 0$ and all Borel sets $A \subset \bbS$ with $\sigma(\partial A) = 0$.
We say that $Y$ is \emph{nondegenerate} if
$\int_\bbS|u\cdot\theta|^\alpha\,d\sigma(\theta)>0$ for all $u\in\bbS$.

Taking $A=\bbS$, we have that $|Y|$ is a regularly varying scalar
function.
Hence there exists a slowly varying function $\ell:[0,\infty)\to(0,\infty)$
such that $\pP(|Y|>t)=t^{-\alpha}\ell(t)$.
\end{defn}

Provided that $(Y_j)_{ j\ge 0}$ is an i.i.d sequence of random
variables in $\R^d$ defined on a probability space $(\Omega_0,\mathcal F_0, \pP_0)$
and that $Y_0$ is regularly varying
of index $-2$ with $\E(Y_0)=0$,
\begin{equation}
\label{eq:law1}
 a_n^{-1}\sum_{j=0}^{n-1} Y_j\implies \cN(0,\Sigma_0^2),\text{ as }n\to\infty,
\end{equation}
 where $\implies$ denotes convergence in distribution, $\Sigma_0$ is a positive definite $d\times d$ matrix
and $a_n$ is given by $\lim_{n\to\infty}\frac{n\tilde\ell(n)}{a_n^2}=1$, where $\tilde\ell(x)=1 + \int_1^{1+x} \frac{\ell(u)}{u} du$.

\begin{rmk}\label{r:Lorentzcase}
Given several important billiard applications (infinite horizon Lorentz gas, Bunimovich stadium, dispersing billiards with cusps), a case of special significance is when $\ell (t)\equiv C$ for some positive constant $C$,
and thus $\tilde\ell(t)\sim C\log t$. We will refer to this as the Lorentz case.
\end{rmk}

We recall the result in~\cite{Ein07,EinLi} for the sequence $(Y_j)_{j\ge 0}$
of independent random variables along with the required assumption.
\begin{itemize}
\item[(H0)] Let $(c_n)_{n\ge 1}$ be a non-decreasing sequence so that
\begin{itemize}
\item[(i)] for any $\eps>0$, there exists $M_\eps\ge 1$ so that
$\frac{c_n}{c_m}\le (1+\eps)\frac{n}{m}$ for all $M_\eps\le m<n$;
\item[(ii)] $\sum_n\pP_0(|Y_0|>c_n)<\infty$.
\end{itemize}
\end{itemize}
\begin{rmk}In the situation of Remark~\ref{r:Lorentzcase}, $\ell(n)\equiv C$ and $\sum_n\pP_0(|Y_0|>a_n)=\infty=\lim_{N\to\infty}\sum_{n\ge N}\frac{C}{a_n^2}$. Since $\sum_n\pP_0(|Y_0|>c_n)<\infty$, $\lim_{N\to\infty}\sum_{n\ge N}\frac{C}{c_n^2}=0$. This can only hold if $\frac{c_n}{a_n}\to\infty$, as $n\to\infty$. As in Corollary~\ref{cor-ein3} below, if the density $f(x)$ of the distribution function
 $\pP(Y_j \leq x)$ satisfies $f(x)\sim C\, |x|^{-3}1_{|x|\ge1}$
 then $c_n=a_n\ell_0(n)$, for some slowly varying function $\ell_0(n)\to\infty$, as $n\to\infty$.
\end{rmk}

\begin{thm}~\cite[Theorem 2]{Ein07},~\cite[Theorem 2.1]{Ein09}.
\label{thm-ein1}
 Let $(Y_j)_{ j\ge 0}$ be an i.i.d sequence of $\R^d$ random
variables, regularly varying with index $-2$, and satisfying~\eqref{eq:law1}. Let $(c_n)_{n \geq 1}$ be a sequence satisfying (H0).

 Then, there exists a probability space $(\Omega_1,\mathcal F_1, \pP_1)$
 and two sequences of independent random vectors
$(Y_j^*)_{ j\ge 0}$, $(Z_j)_{ j\ge 0}$ so that
\begin{itemize}
 \item $Y_j^*$ is distributed as $Y_j$
 \item $Z_j$ is distributed as $\cN(0,I)$
\end{itemize}
 and so that, almost surely, as $n\to\infty$,
 \[
   \left|\sum_{j=0}^{n-1} Y_j- \Gamma_n\,\sum_{j=0}^{n-1} Z_j\right|=o(c_n),
  \]
where the matrix $\Gamma_n$ is given by $\Gamma_n^2=cov(Y_0\,1_{|Y_0|<c_n})$.
 \end{thm}

When compared to the standard ASIP, a specific feature of Theorem~\ref{thm-ein1} is the rescaling of Brownian motion with the truncated covariance matrix
$\Gamma_n$, which occurs only for variables in the non-standard domain of attraction of the normal law. This reflects the fact that in this regime the limit phenomena are particularly
sensitive to the specifics of the tail of the distribution, which typically makes the treatment quite challenging.  Indeed, for SIP type results for i.i.d. random variables in the domain of a symmetric
stable law of index $\alpha\in (0,2)$ (see \cite[Theorems 3.1 and 3.2]{St79}) there is no need for truncation and thus, no need for an analogue
of $\Gamma_n$, and the main challenge is to obtain good rates (we refer to \cite[Section 3]{St79}). Note also that in this $\alpha<2$ case the approximation is, not by sums of independent Gaussians, but instead by sums of independent $\alpha$-stable variables, and thus there is no natural analogue of the sequence $\sqrt{n L(L(n))}$ in \eqref{eq:standardlil} (for the standard CLT case), or the sequence \eqref{eq:cnlg} below (in our case), see
\cite[Theorem 2.1]{St79}.  Thus, although as in the case of Theorem~\ref{thm-ein1}, there is infinite variance in the $\alpha<2$ stable law scenario too, the limit phenomena are quite different,
see also \cite{DolgLiu} for recent results in the dynamical context.

As in~\cite{EinLi, Ein07} the main purpose of
Theorem~\ref{thm-ein1} is the implication of a precise form of a ``generalized law of the iterated logarithm (LIL)'' for such i.i.d. processes, which we now recall.
\begin{thm}~\cite[Theorem 3]{EinLi},~\cite[Theorem 1]{Ein07}.
\label{thm-ein2}Assume the set up of Theorem~\ref{thm-ein1}.
 Set
 \begin{align*}
A(t)&=\sup \left\{ \E[\langle Y_0, y \rangle^2 1_{|Y_0|\le t}]; y\in \R^d, |y|\le 1\right\};\\
a&=\sup\left\{\alpha\ge 0: \sum_{n=1}^\infty \frac{1}{n}\exp\left(-\frac{\alpha^2c_n^2}{2n A(c_n)}\right)=\infty\right\}.
 \end{align*}
Then almost surely,
 \[
   \limsup_{n\to\infty}\frac{|\sum_{j=0}^{n-1} Y_j|}{c_n}=a.
  \]
 \end{thm}
 The definition of $A(t)$ in Theorem~\ref{thm-ein2} is exactly as in
 ~\cite[Theorem 1]{Ein07}, since $\sup \left\{ \langle w, y \rangle^2; |y|\le 1 \right\}=|w|^2$ for any vector $w\in\R^d$, $A(t)=\E(|Y_0|^2 1_{\{|Y_0|\le t\}})$.

Note that in Theorem~\ref{thm-ein2}, the limit $a$ can be $0$, infinite, or a positive finite number depending on the sequence $c_n$, and it is not always easy to say for which sequences we have $a\in(0,\infty)$, a situation which can be regarded as a ``generalized law of the iterated logarithm''.
In special cases both the limit $a$ and the sequence $(c_n)_{n \geq 1}$ can be described precisely:
\begin{cor}
\label{cor-ein3}
Consider an i.i.d.~sequence of $\R$-valued, symmetric random variables with tail behaviour as in Theorems~\ref{thm-ein1} and~\ref{thm-ein2}.
Assume, furthermore, that the density $f(x)$ of the distribution function
 $\pP(Y_j \leq x)$ satisfies $f(x)\sim C\, |x|^{-3}1_{|x|\ge1}$, for some $C>0$. Then (recall Notation~\ref{n:LLL}), according to
 \cite[p.~1622]{EinLi},
 \begin{align}\label{eq:cnlg}
 \text{If}
  &\ c^*_n=\sqrt{2C\cdot n\cdot L(n)\cdot LL(n) \cdot (1+ LL(n) \sin^2(LLL(n)) ) } \\
 \text{then }
\label{eq:a} &\limsup_{n\to\infty}\frac{|\sum_{j=0}^{n-1} Y_j|}{c^*_n}=1; \text{ almost surely}.
 \end{align}

\end{cor}
The argument in~\cite[p.~1622]{EinLi} is given for $C=1$, but it is straightforward to
get the corresponding statement for $C\ne 1$.

\subsection{Main result: generalized ASIP and LIL for the Lorentz gas map with infinite horizon}
\label{sec:lg}

Let us recall the terminology and notations form section~\ref{ss:lgintro}. An important property of the cell function $\kappa$, which we shall exploit below, is
precise information on the corresponding density function.
More precisely, we recall from~\cite[Proposition 6]{SV07} that there exists a finite collection $\mathop{Corr}\subset \Z^2$ of vectors (the corridor vectors),
such that, for any $\xi\in\mathop{Corr}$, there exists $C_{\xi}>0$, such that (as $N\to\infty$)
\begin{align}\label{densk}
 \mu(\kappa= N \xi)\sim C_{\xi}\cdot N^{-3}.
\end{align}
Here the event $\kappa = N \xi$
indicates the next collision with a scatterer at distance $|N||\xi|$ from the current position. We note that this implies that $\ell(t)\equiv C$ for some $C>0$ (recall Remark~\ref{r:Lorentzcase}).

The result below is a version of Theorem~\ref{thm-ein1} (a certain ASIP) and Theorem~\ref{thm-ein2} (generalized LIL) for $\kappa$ and $\Phi$.

\begin{thm}
 \label{thm-LG}
 Consider the Sina\u{\i} billiard map $(M,T,\mu)$ with infinite horizon
 and either the cell-change function $\kappa:M\to\Z^d$ or the flight function
 $\Phi:M\to\R^d$.

 Let $c^*_n$ be defined as in~\eqref{eq:cnlg} (with exactly the same constant $C$ as in Remark~\ref{r:Lorentzcase}). Then

 \begin{itemize}
  \item[(i)]
 There exist a probability space $(\Omega^*,\mathcal F^*, \pP^*)$
 and two sequences of random vectors
$(v_j^*)_{ j\ge 0}$, $(Z_j)_{ j\ge 0}$ so that
\begin{itemize}
 \item[$\bullet$] $v_j^*$ is distributed either as $\kappa\circ T^j$ or
 as $\Phi\circ T^j$.
 \item[$\bullet\bullet$] the vectors $Z_j, j\ge 0$ are independent and  distributed as $\cN(0,I)$;
\end{itemize}
 such that, almost surely, as $n\to\infty$,
 \[
   \left|\sum_{j=0}^{n-1} v_j^*- \Gamma_n  \,\sum_{j=0}^{n-1} Z_j\right|=o(c^*_n),
  \]
where the matrix $\Gamma_n$ is given by $\Gamma_n^2=C\cdot \log(c^*_n)\,\Sigma^2$.\footnote{Actually, statement (i) of Theorem~\ref{thm-LG} holds not
only for $c_n^*$, but extends also to a more general class of sequences $c_n$, see condition (HL0) in section~\ref{sec:abstrsetup} below.}

 \item[(ii)]\(
   \limsup_{n\to\infty}\frac{\left|\sum_{j=0}^{n=1}\kappa\circ T^j\right|}{c^*_n}=a,
  \) almost surely, for some $a\in(0,\infty)$, and a similar statement holds for $\Phi$.
 \end{itemize}~\end{thm}


In Section~\ref{sec:abstrsetup} we state the main results needed for the proof of Theorem~\ref{thm-LG}, namely Propositions~\ref{thm-main} and~\ref{thm-main2}. These Propositions state, actually, an abstract version of the result, which will make the application
to the Lorentz gas with infinite horizon easier. Along with stating these key Propositions we also describe briefly the structure of the paper in section~\ref{sec:abstrsetup} below.

\section{Main ingredients for the proof of Theorem~\ref{thm-LG}}
\label{sec:abstrsetup}

In this section we phrase a set of  assumptions in terms of dynamical systems
 and observables associated with them. The abstract assumptions we consider below are designed to satisfy the set up of the Lorentz gas.

 Before stating the precise assumptions, let us give a summary
 of the ingredients we shall use in the paper.
 Besides Einmahl's results (in particular \cite{Ein09}) on the i.i.d case, two further important references for the proof of Theorem~\ref{thm-LG} are \cite{Gou-asip} and \cite{ChDo09}.

In \cite{Gou-asip}, Gou\"ezel proved the ASIP for a large class of processes with $Y\in L^p$ for some $p>2$. In our argument, we adapt Gou\"ezel's strategy, specifically with $p=4$,
to our situation after truncating the variables $Y_i$. The choice of the truncation level is delicate and the method requires very precise bounds on the second and the fourth moment for partial sums
of the appropriately
truncated random variables. Although in \cite{ChDo09} Chernov and Dolgopyat obtained moment bounds of similar character in the infinite horizon Lorentz gas, we improve significantly upon these bounds, in particular,
we manage to remove some logarithmic factors from the bound on the fourth moment, obtaining this way a scaling essentially as good as for i.i.d.
random variables, as expressed in condition (H3)(ii) below. Another important feature of our argument is that, additionally to the truncation level $c_n$, we introduce slightly lower truncation levels at $d_n$ and $\bd_n$, see Formula \eqref{eq:dn} and conditions (HL1-HL3) for details. This is partially inspired by \cite{ChDo09} which also uses further truncation levels, yet, this idea has to be implemented at the level of almost sure convergence
in our case.

The proof of Proposition~\ref{thm-main} is broken down into several Propositions and Lemmas in sections~\ref{sec:main},~\ref{sec:mainprop} and~\ref{sec:mainlem} -- this
is essentially an adaptation of Gou\"ezel's argument from \cite{Gou-asip}, which nonetheless exploits the properties of our variables -- expressed in the conditions of Proposition~\ref{thm-main} -- in a subtle way. These conditions are verified for the infinite horizon Lorentz gas in sections~\ref{sec:apl} and~\ref{sec:fm}. In particular, the above mentioned improvement on the bound of the fourth moment of partial truncated sums for the infinite horizon Lorentz gas, an interesting result on its own, is proved in Section~\ref{sec:fm}, and conditions (HL1-HL3) on the truncation levels $c_n$, $d_n$ and $\bd_n$ are verified in Section~\ref{sec:fm}, too. In the Appendix we check that the sequence \eqref{eq:cnlg} satisfies the required summability condition of (HL0), completing this way the proof of our ``generalized law of the iterated logarithm'' ie.~Theorem~\ref{thm-LG}(ii).

\subsection{Set up}
 \label{subsec:abstrsetup}

 We consider an invariant  probability measure preserving transformation $(\Omega,T, \mu)$.

Let $v: \Omega\to \R^d$  and set $v_j=v\circ T^j$. Below we formulate the assumptions (H1)-(H3) and (HL1)-(HL3) which will be used in our proofs and verified for the infinite horizon
Lorentz gas in the bulk of the paper.
Nonetheless, we include Remark~\ref{r:whyforLG} below to comment on the relevance of these assumptions for the infinite horizon
Lorentz gas.

Throughout, we assume
\begin{itemize}
\item[(H1)]
 $\E_\mu(v)=0$,  $v$ is regularly varying with index $-2$, let $\ell$ be the involved slowly varying function and require that
 \begin{equation}
\label{eq:lawd}
 a_n^{-1}\sum_{j=0}^{n-1} v_j\implies \cN(0,\Sigma^2),\text{ as }n\to\infty,
\end{equation}
 where $\Sigma$ is a positive definite $d\times d$ matrix
and where $a_n$ is so that $\lim_{n\to\infty}\frac{n\ell^*(n)}{a_n^2}=1$, where the slowly varying function $\ell^*$ satisfies $\ell^*(t)=C_0\tilde\ell(t)$, for some constant $C_0>0$. (Recall that $\tilde\ell(x)=1 + \int_1^{1+x} \frac{\ell(u)}{u} du$.)
\end{itemize}

\begin{rmk}\label{r:llstarLorentz}
In the Lorentz case billiard examples, the functions $\tilde\ell$ and $\ell^*$ are as follows. For the infinite horizon Lorentz gas, we have $\ell^*(t)=\tilde\ell(t)=C\log t$, with some $C>0$, that is,
we have $C_0=1$ and observe the same limit law as for iid sequences with the same tail behaviour (see \cite{BCDcusp}, \cite{ChDo09}). For Bunimovich stadia, we have $\ell^*(t)=C_0\tilde\ell(t)=C_0C\log t$ with $C_0=\frac{4+3\log 3}{4-3\log 3}$. Thus, the relevant slowly varying function differs from that of the iid case by a nontrivial constant factor, an effect which is due to short range correlations between consecutive large values (see \cite{BalintGouezel06}).
\end{rmk}

We are interested in the almost sure behaviour of $\sum_{j=0}^{n-1}v_j$,
in the sense of an analogue of Theorem~\ref{thm-ein1}, and with this
obtain an analogue of Theorem~\ref{thm-ein2}.

The remainder of our assumptions will be phrased in terms of the following truncated version of $v$. For  $R\in\R_{+}\cup\{\infty\}$ (and for any $l\ge 0$, consider also $R(=R_l)\in\R_{+}\cup\{\infty\}$) define
 \begin{align}\label{eq:wjdn}
W^R=v\cdot\mathbf{1}_{|v|\le R} \qquad \text{and accordingly} \qquad W_l^{R}=W^R\circ T^l =v_l\cdot\mathbf{1}_{|v_l|\le R}.
 \end{align}

To ensure `sufficiently weak dependence', we assume~\cite[Assumption(H)]{Gou-asip} at the level of characteristic functions of $n+m$ blocks derived from the truncated variables \eqref{eq:wjdn}, with gaps of size $k$ in between the two blocks. More precisely:
\begin{itemize} \item[(H2)]
There exist $\eps_0>0$, $C, c>0$ so that for any $n,m>0$, $b_1<b_2<\ldots< b_{r+m+1}$, for any $k>0$ and for all $t_1,t_2,\ldots, t_{n+m}\in\R^d$ with $|t_j|\le \eps_0$, and for arbitrary truncation levels $R_1,R_2,\ldots, R_{n+m}\in\R^+$
\begin{align*}\Big|\E_\mu &\left(e^{i\sum_{j=1}^n t_j(\sum_{l=b_j}^{b_{j+1}-1} W_l^{R_j})+i\sum_{j=n+1}^{n+m} t_j(\sum_{l=b_{j}+k}^{b_{j+1}+k-1} W_l^{R_j})} \right) \\
 & - \E_\mu\left(e^{i\sum_{j=1}^n t_j(\sum_{l=b_j}^{b_{j+1}-1} W_l^{R_j})} \right) \cdot \E_\mu\left(e^{i\sum_{j=r+1}^{r+m} t_j(\sum_{l=b_{j}+k}^{b_{j+1}+k-1} W_l^{R_j})}\right)\Big|\\
 &\le C(1+\max_{0\le j\le m+n}|b_{j+1}-b_j|)^{C(m+n)} e^{-ck}.
\end{align*}
\end{itemize}
In (H2) and throughout, we write $e^{itx}$ instead of $e^{i \langle t, x\rangle}$. Note, furthermore, that the variable $t_j$ in the characteristic function corresponds to the sum of the variables $W_l^{R_j}$ for
$b_j\le l<b_{j+1}$, which are all truncated at the same level $R_j$.

 We further assume some control of moments of sums of truncated observables
 along with a precise form of the covariance matrix.
\begin{itemize} \item[(H3)]
Let $W_j^{R}$ be as in~\eqref{eq:wjdn} with $R_j=R$ for all $j$.
We require that
\begin{itemize}
\item[(i)] $\E_\mu(W_j^R)=0$, for any $R$.
\end{itemize}
For both (ii) and (iii), consider $m\in \N$ (the number of terms) and $0<R(=R(m))$ (a truncation level), with $R\to \infty$ and $m\to \infty$, such that $m=o(R^2)$.
\begin{itemize}
\item[(ii)] Additionally, restricting to $m=o(R^{2-r_1})$ and $R=o(m^{r_2})$ for some $r_1>0$ and $r_2>0$,
we assume that\footnote{We will
always apply (H3)(ii) in situations when $m=2^{\alpha n}$ and $R=c_{2^n}$, where $n\to \infty$, $\alpha\in(0,1)$ and the sequence $c_k$  (in accordance with (H0)) is equal to $\sqrt{k}$
up to slowly varying factors. Hence $c_{2^n}$ is equal to $2^{n/2}$ up to factors that grow polynomially in $n$, and the requirements on the relation of $m$ and $R$ are satisfied. }  there exists $C>0$ so that
\(
\int \left|\sum_{j=0}^{m-1}  W_{j}^{R}\right|^{4}\, d\mu\le  C\, m\,R^{2}.
\)
\item[(iii)] Assume there exists $C>0$ so that $cov (\sum_{j=0}^{m-1} W_j^{R})= Cm\ell^*(R)\,\Sigma^2(1+o(1))$, where $\Sigma$ and $\ell^*$ are as in~\eqref{eq:lawd}.
\end{itemize}

\end{itemize}

We obtain a non iid version of Theorem~\ref{thm-ein1} for a slightly more restrictive
class of sequences $c_n$ than the one in  (H0).
For concreteness, we specialize to the Lorentz case (see Remark~\ref{r:Lorentzcase}), which in particular implies $\ell^*(R)\asymp \tilde\ell(R)\asymp L(R)$
(see conditions (H1) and (H3) above).
Concerning the class of sequences $c_n$, we make the following assumptions.

\begin{itemize} \item[(HL0)] $c_n=\sqrt{n \cdot \ell^* (n) \cdot \ell_1 (n)} \asymp \sqrt{n \cdot L(n) \cdot \ell_1 (n)}$ for some slowly varying function $\ell_1 (n)$, such that the following conditions hold.
\begin{itemize}
\item[(i)] There exists some $b>1$ such that $\ell_1(n)\ll (LL(n))^b$ (note this assumption does not weaken our results as it puts an upper bound on $c_n$)
\item[(ii)] $c_n$ satisfies (H0). As we assume Lorentz case, and $c_n$ is of the above form, this is equivalent to $\sum_n c_n^{-2} = \sum_n (n L(n) \ell_1(n))^{-1}<\infty$, in accordance with the previous item.
\item[(iii)] Our additional summability condition is formulated in terms of $c_{2^n}$. It is worth noting that $L(2^n)  \asymp n$ while $\ell_1(2^n) \ll (L(n))^b$, with $b$ as in the item (i) above. We assume $\sum_n \frac{2^n LL(n)}{(c_{2^n})^2}<\infty$, or equivalently  $\sum_n \frac{LL(n)}{n \ell_1(2^n)}<\infty$.
\end{itemize}
\end{itemize}

\begin{rmk}\label{rmkcn}
The  sequence $c_n$ defined in~\eqref{eq:cnlg} satisfies (HL0): while (i) is immediate and (ii) is proved in \cite{EinLi}, item (iii) requires a justification: see Appendix~\ref{sec:cn}.
Let us point out that we will also use explicitly the following bound, slightly weaker than (iii):
\begin{equation}\label{eq:HL0iiiw}
  \sum_n \frac{2^n}{(c_{2^n})^2}= \sum_n \frac{1}{n \ell_1(2^n)}<\infty;
\end{equation}
and note that, for the  sequence $c_n$ defined in~\eqref{eq:cnlg}, \eqref{eq:HL0iiiw} is proved in~\cite[p.~1622]{EinLi}.
\end{rmk}

We make some further assumptions that allow us to simplify the exposition.
To formulate them, we introduce two additional truncation levels,  $d_n$ and $\bd_n$, lower than $c_n$ (which is assumed to have the form specified in (HL0)). More precisely, let
\begin{align}\label{eq:dn}
 d_n &=\sqrt{n L(n) \ell_1(n)^{-99}}, \text{ while}\\
 \nonumber\bd_n& =n^{1/2-\varsigma}, \text{ for some small }\varsigma>0, \text{ to be specified later}.
 \end{align}
That is, the truncation level $d_n$ is just slightly smaller than $c_n$, while $\bd_n$ is even smaller.
(The details on choice of $\varsigma$ are specified inside the proof of Lemma~\ref{lemma:max} below.)

\begin{itemize} \item[(HL1)] This assumption concerns values between the truncation levels $d_{2^n}$ and $c_{2^n}$. Let $n\ge 1$ and set $\tW= v\cdot \mathbf{1}_{d_{2^n}\le|v|\le c_{2^n}}$.
Fix $r\ge 1$ and define $\tW_{l}=\tW\circ T^{l+r}$ for $l\ge 1$.
Fix $N\le (c_{2^n})^2$ and let $\tcS_N=\sum_{l=0}^{N} \tW_l$. We assume there exists $C_1>0$ such that:
\[
\E_\mu\left(\tcS_N^2\right)\le C_1 \cdot N\cdot L\left(\frac{c_n}{d_n}\right)\ll N\cdot LL(n).
\]

\item[(HL2)] This assumption concerns values between the truncation levels $\bd_{2^n}$ and $c_{2^n}$.
Let $n\ge 1$ and set $\bW= v\cdot \mathbf{1}_{\bd_{2^n}\le|v|\le c_{2^n}}$.
Fix $r\ge 1$ and define $\bW_{l}=\bW\circ T^{l+r}$ for $l\ge 1$.
Fix $K\in\N$ with $K< (\bd_{2^n})^2$. For $k=1,\dots, K$ and $\eps>0$, define
\[
\cS_k=\sum_{l=0}^{k} \bW_l;\qquad B_{k,\eps}=\{x\in M \,:\, |\cS_k|\ge \eps c_{2^n}\}
\]
and let
\(
N_{k,\eps}=B_{k,\eps}\setminus (\cup_{k'=0}^{k-1} B_{k',\eps}).
\)

We assume that there exists $C_2>0$ so that
 \[ \sum\limits_{k=1}^K
\mu(N_{k,\eps} \setminus B_{K,\eps/2})\le C_2 \frac{K}{(c_{2^n})^2}.
\]

 \item[(HL3)] This assumption concerns values lower than the truncation level $\bd_{2^n}$. Let $W_l^R$ as in~\eqref{eq:wjdn} and $\bd_{2^n}$ as above, and, for $k\ge 0$,  let
 \begin{equation}\label{eq:sigma_algebras}
 \mathcal{F}_{0,k}(W^{\bd_{2^n}}) =\sigma(W_0^{\bd_{2^n}},W_1^{\bd_{2^n}},\dots,W_k^{\bd_{2^n}})
 \end{equation}
 denote the sigma algebra generated by the random variables $W_0^{\bd_{2^n}},W_1^{\bd_{2^n}},\dots,W_k^{\bd_{2^n}}$. We assume that,
given $q(n)=\delta_0\frac{c_{2^n}}{\bd_{2^n}}$
 with $\delta_0\in (0,1)$,
 \[
  \left\|\E_\mu\left(W_{k+q(n)}^{\bd_{2^n}}\Big|\mathcal{F}_{0,k}(W^{\bd_{2^n}})\right)\right\|_{L^1(\mu)}\le C \left(\bd_{2^n} \gamma^{q(n)-C_3\log(\bd_{2^n})}  + (\bd_{2^n})^{-10}\right),
 \]
 for some $\gamma\in (0,1)$
 and for some $C>0$, $C_3>0$, uniformly in $k$.

\end{itemize}

\begin{rmk}\label{r:whyforLG}
Assumptions (H1)-(H3) and (HL1)-(HL3) may seem a bit unusual at first sight, nevertheless, these quite naturally reflect the following properties of the infinite horizon Lorentz gas.
\begin{description}
\item \textit{Tail behaviour.} The cell change function $\kappa$ (and thus the flight function $\Phi$) have tail distributions as expressed in (H1), actually, with the specific slowly varying functions as in Remark~\ref{r:Lorentzcase}.
\item \textit{Weak dependence.} The values of $\kappa\circ T^i$ decorrelate rapidly in $i$, which is expressed for example in (H2). Conditions (H3)(ii) and (H3(iii) also follow for the Lorentz gas  from the specific tail behaviour and the weak form of dependence. In particular, with the tail behaviour of Remark~\ref{r:Lorentzcase}, $\mathbb{E}_{\mu}(v^4\cdot \mathbf{1}_{|v|<R})\asymp R^2$, hence (H3)(ii) states (taking into account also the relations of $m$ and $R$) that the fourth moment in this case essentially behaves as if the variables were i.i.d.
\item \textit{Symmetry of the distribution.} In the Lorentz gas, by symmetry, $\kappa$ and $-\kappa$ have the same distribution (and the same holds also for $\Phi$). This naturally implies that the mean of all truncated variables vanish, as expressed in (H3)(i). We believe that this assumption is not essential for our method, and could be relaxed with some extra work. However, we do not push this point here to avoid further increase of the length of the paper.
\end{description}
The further technical conditions (HL1)-(HL3) also reflect these properties. Finally, we note hat the above three features (tail behaviour, weak dependence and symmetric distribution) are relevant not only for the Lorentz gas, but also for other examples of mechanical origin (for example, stadia or dispersing billiards with cusps), thus we hope that our scheme can be extended to such further cases, but do not study this in the present paper.
\end{rmk}

\subsection{Main propositions used in the proof of Theorem~\ref{thm-LG}}

The first result below is an analogue of Theorem~\ref{thm-ein1} in the non independent set up of subsection~\ref{subsec:abstrsetup}.

 \begin{prop}
 \label{thm-main}
  Assume (H1)--(H3). Let $(c_n)_{n \geq 1}$ be a sequence satisfying (HL0).
  Suppose that conditions (HL1)-(HL3) hold.
  Then, there exist a probability space $(\Omega^*,\mathcal F^*, \pP^*)$
 and two sequences of random vectors
$(v_j^*)_{ j\ge 0}$, $(Z_j)_{ j\ge 0}$ so that
\begin{itemize}
 \item $v_j^*$ is distributed as $v_j$;
 \item the vectors $Z_j, j\ge 0$ are independent and  distributed as $\cN(0,I)$;
\end{itemize}
 such that, almost surely, as $n\to\infty$,
 \[
   \left|\sum_{j=0}^{n-1} v_j^*- \Gamma_n  \,\sum_{j=0}^{n-1} Z_j\right|=o(c_n),
  \]
where the matrix $\Gamma_n$ is given by $\Gamma_n^2=\ell^*(c_n)\,\Sigma^2$.\\
Note that $\sum_{j=0}^{n-1} Z_j=W(n)$, where $W(n)$ is standard Brownian motion in dimension $d$, at time $n$.
\end{prop}

The proof of Proposition~\ref{thm-main} is provided in Section~\ref{sec:main}. It relies on a combination of arguments/results in~\cite{Gou-asip} under (H2) and (H3) together with arguments/results in~\cite{Ein07} (some
established already in previous works by the same author).

With exactly the same proof as in~\cite{Ein09}, Proposition~\ref{thm-main}
gives the Generalized LIL (the analogue of Theorem~\ref{thm-ein2}):

\begin{prop}
 \label{thm-main2}
 Assume (H1)--(H3). Let $(c_n)_{n \geq 1}$ be a sequence satisfying (HL0)
 and suppose that conditions (HL1)-(HL3) hold.
  Set
 \[
 A(t)=\ell^*(t) \|\Sigma\|^2,\,\,\,a=\sup\left\{\alpha\ge 0: \sum_{n=1}^\infty \frac{1}{n}\exp\left(-\frac{\alpha^2c_n^2}{2n A(c_n)}\right)=\infty\right\}.
\]
  Then almost surely,
 \[
   \limsup_{n\to\infty}\frac{\left|\sum_{j=0}^{n=1}v_j\right|}{c_n}=a.
  \]
\end{prop}
In statement of Proposition~\ref{thm-main2}, we use the operator norm $ \|A\|= \sup\{|Av|:|v|\le 1 \}$ as in~\cite[page 20]{Ein09}. With this norm, $A(c_n)=\|\Gamma_n\|$.

While  the proof of Theorem~\ref{thm-ein1} in~\cite{Ein09} heavily relies on the
i.i.d. structure, the proof of Theorem~\ref{thm-ein2} given Theorem~\ref{thm-ein1} does not require
independence. In the same way, Proposition~\ref{thm-main} gives Proposition~\ref{thm-main2}.
Indeed, given Proposition~\ref{thm-main}, Proposition~\ref{thm-main2} reduces to
\[
 \limsup_{n\to\infty}\frac{\left|\Gamma_n \cdot W(n)\right|}{c_n}=a,
\]
which is proved in \cite[proof of Corollary 2.4]{Ein09} in \cite[section 4.2.1]{Ein09} and \cite[section 4.2.2]{Ein09}, cf.~in particular
\cite[Formulas (4.4)-(4.5)]{Ein09} and \cite[Formulas (4.8)-(4.9)]{Ein09}.

Thus, our only task is to prove Proposition~\ref{thm-main}. Proposition~\ref{thm-main}
is the main technical contribution of this paper and its proof is provided in Section~\ref{sec:main}.

\vspace{-2ex}
\paragraph{Notations.}
Throughout the proofs, we use ``big O'' and $\ll$ notation interchangeably, writing $b_n=O(c_n)$ or $b_n\ll c_n$
if there are constants $C>0$, $n_0\ge1$ such that
$b_n\le Cc_n$ for all $n\ge n_0$. Furthermore, $b_n\asymp c_n$ means that $b_n\ll c_n$ and $c_n\ll b_n$ hold simultaneously. This is not to be confused with the stronger property
$b_n \sim c_n$ which means that $b_n= c_n \cdot (1+o(1))$.
\\

\section{Proof of Theorem~\ref{thm-LG}}
\label{sec:apl}

In this section we derive Theorem~\ref{thm-LG}
from Propositions~\ref{thm-main} and~\ref{thm-main2} with $c_n$ be defined as in~\eqref{eq:cnlg}.

We already know that such $c_n$ satisfies (HL0): see Remark~\ref{rmkcn}.
Thus, to obtain (i) of Theorem~\ref{thm-LG}, we need to verify that assumptions (H1)--(H3) hold. In this scenario, the probability measure preserving map $(\Omega,T,\mu)$
  is replaced by $(M,T,\mu)$ and
  the truncated observable $W_l^{R_l}$ as in~\eqref{eq:wjdn} is defined with $\kappa$ (respectively, $\Phi$) instead of $v$.
  It is known that $\kappa$ and $\Phi$ satisfy (H1), that is they are regularly varying observables with index $-2$ and as shown in~\cite{SV07} (via the Young tower constructed in~\cite{Chernov99,
 Young98}),
 $\kappa$ and $\Phi$ satisfy~\eqref{eq:lawd} with $a_n=\sqrt{n\log n}$.
 It is known that $\mu(\kappa)=0$ and (H3)(i) holds by symmetry.
 (H3)(iii) is known to hold with $C=2$ and $\ell^*=\log$ (see first displayed formula in ~\cite[Proof of Lemma 11.3]{ChDo09}).
 (H3)(ii)  and (HL1)-(HL3)) are verified in Section~\ref{sec:fm}. (H2) is verified below.\\

Item (ii) of Theorem~\ref{thm-LG} follows from Proposition~\ref{thm-main2}.
The precise description of $c_n$ and limit $a$ is possible because
using~\eqref{densk} we obtain that $A(t)=C L(t)\|\Sigma\|^2$; here $C$ is the same as in~\eqref{densk}.
Given $c_n$  as in~\eqref{eq:cnlg} (with exactly this positive $C$),
it is easy to see that  $a=1$.\\

 \textbf{Verifying (H2).}
We first recall a fact established in~\cite{Gou-asip}, which shows that (H2)
holds under very mild assumptions on the transfer operator.

Let $T:\Omega\to\Omega$ be a non-singular transformation w.r.t.\, $m=Leb$ and let $L:L^1(m)\to L^1(m)$  be the transfer operator given by $\int_\Omega Lf\,g\, dm=\int_\Omega f\,g\circ T\, dm$, for $f\in L^1(m)$ and $g$ bounded and measurable.

Now we define perturbations of the transfer operator $L$ associated to our observable $v:\Omega\to\R^d$ (satisfying (H1)) and its truncations $W^R$ (as defined in Formula \eqref{eq:wjdn}).
Given a perturbation parameter $t\in\R^d$ and a truncation level $R\in\R^+\cup\{\infty\}$, let
\[
L_{t,R} f=L(e^{itW^R} f).
\]
For $t=0$ and arbitrary $R$, we apparently have $L_{0,R}=L_0=L$.

The main difference between our (H2) and \cite[Condition (H)]{Gou-asip} is that (H2) involves several truncation levels, but this does not create a problem in checking that (H2) holds under mild conditions
on the perturbed operators $L_{t,R}$. Indeed,
we can proceed essentially as in~\cite[Section 2]{Gou-asip}.

To start, let us slightly generalize the following notion from \cite[Formula (2.2)]{Gou-asip}.

\begin{defn} \label{def:code}
Recall the setting and the notation of section~\ref{sec:abstrsetup}, in particular, let $\mu$ be an invariant measure for $T$.
Suppose that there exists  a Banach space $\mathcal B$ (of functions or distributions), and moreover, some $u_0\in\cB$ and $\xi_0\in\cB'$ (the dual of $\cB$) so that for any $t_l\in\R^d$, $|t_l|\le\eps_0$ with $0\le l\le n-1$, and any $R_l\in\R^+\cup\{\infty\}$, again with $0\le l\le n-1$, we have
 \begin{align}\label{eq:exprtri}
  \E_\mu(e^{i\sum_{l=0}^{n-1}t_l W_l^{R_l}})=
  \langle\xi_0,L_{t_{n-1},R_{n-1}}L_{t_{n-2}, R_{n-2}}\ldots L_{t_1, R_1} L_{t_0, R_0}u_0\rangle.
 \end{align}
 In this case we say that the random variables $v\circ T^l$  along with their  truncations $W_l^{R_l}$ are coded by $(\cB,L_{t,R},u_0,\xi_0)$.
\end{defn}

\begin{lemma}{~\cite[Section 2]{Gou-asip}}\label{lemma:wcontin} Suppose that the random variables $v\circ T^l$ are coded by some $(\cB,L_{t,R},u_0,\xi_0)$ in the sense above.
Assume, furthermore, that the family of operators $L$ and $L_{t,R}$ satisfy the properties (I1) and (I2) below.
 \begin{enumerate}
  \item[(I1)] There exists $C>0$ and $\delta_0<1$ such that $L(=L_0)$ as an operator acting on $\cB$  satisfies $L=\Pi+Q$ where $\Pi$ is a one dimensional projection, while $\Pi Q=Q \Pi$ and for any $n\ge 1$
  $\|Q^n\|_{\cB\to\cB}\le C\delta_0^n$.
  \item[(I2)]  There exists $C'>0$ and $\eps_0>0$ such that for all $|t|\le\eps_0$, all $R\in\R^+\cup\{\infty\}$, and all $n\in\N$ we have $\|L_{t,R}^n\|_{\cB\to\cB}\le C'$.
 \end{enumerate}
 Then (H2) holds.\\
 Moreover, instead (I2) it suffices to check the following (strong) continuity condition.
 \begin{enumerate}
 \item[(I2c)] The family $L_{t,R}$ depends, in the operator norm of the space $\cB$, continuously on the parameter $t\in\R^d$ at $t=0$, uniformly in the truncation level $R$. That is $\|L_{t,R}-L\|_{\cB\to\cB}\to 0$ as $t\to 0$, uniformly in the truncation level $R$.
\end{enumerate}
\end{lemma}

\begin{proof}
Let us note first that conditions (I1) and (I2c) together imply condition (I2). This is proved for one parameter families $L_t$ in \cite[Proposition 2.3]{Gou-asip}.
To see that this extends to two parameter families $L_{t,R}$ if the continuity is uniform in $R$, recall that the essence of the argument is that under (I1) and (I2c),
there exists some decomposition $L_{t_R}=\lambda_{t,R} \Pi_{t,R} + Q_{t,R}$ such that $\|Q_{t,R}^n\|\le C\delta_0^n$ for small enough $t$, uniformly in $n$ and $R$, $\Pi_{t,R}$
is a one dimensional projection that limits to $\Pi$ as $t\to 0$, and thus $|\lambda_{t,R}|\le 1$ for small $t$, again uniformly in $R$.

Now let us proceed to the verification of (H2) under (I1) and (I2). To simplify the exposition, we introduce the following shorthand notations. Keeping in mind that for any fixed $t_j$, ($j=1,\dots , n+m$), in (H2) we consider the sum of the random variables $W_l^{R_j}$, with $b_j\le l <b_{j+1}$, all of which are truncated at the same level $R_j$, let
\[
L_j =L_{t_j,R_j};\qquad \text{i.e.}\qquad L_jf=L(e^{it_jW^{R_j}} f) \qquad \text{and}\quad e_j=b_{j+1}-b_j.
\]
We record the following immediate consequence of (I2):
\begin{equation}\label{eq:ljdj}
\|L_j^{e_j}\|_{\cB\to\cB}\le C', \quad \text{uniformly in } e_j.
\end{equation}

As the characteristic functions are coded by our perturbed transfer operators in the sense of Definition~\ref{def:code}, using the above notations, the left hand side of (H2) can be written as
\begin{align*}
 &\E_\mu\left( e^{i \sum_{j=1}^n t_j \sum_{l=b_j}^{b_{j+1}-1} W_{l}^{R_{j}} + i \sum_{j=n+1}^{n+m} t_j \sum_{l=b_j+k}^{b_{j+1}+k-1} W_{l}^{R_{j}} } \right) \\
 &= \left\langle \xi_0 \ , \
 L_{n+m}^{e_{n+m}}
 \cdots
  L_{n+1}^{e_{n+1}}
  L^k
   L_{n}^{e_{n}}
   \cdots
   L_{1}^{e_1}
 u_0 \right\rangle \\
 &= \left\langle \xi_0 \ , \
 L_{n+m}^{e_{n+m}}
 \cdots
  L_{n+1}^{e_{n+1}}
  (L^k-\Pi)
   L_{n}^{e_{n}}
   \cdots
   L_{1}^{e_1}
 u_0 \right\rangle\\
 & + \left\langle \xi_0 \ , \
 L_{n+m}^{e_{n+m}}
 \cdots
  L_{n+1}^{e_{n+1}}
  (\Pi)
   L_{n}^{e_{n}}
   \cdots
   L_{1}^{e_1}
 u_0 \right\rangle.
\end{align*}

This together with the bounds in (I1) (in particular $\|Q^n\|_{\cB\to\cB}\le C \delta_0^n$) and \eqref{eq:ljdj},
allows us to proceed exactly the same way as in~\cite[Step 2, (H) holds, Subsection 2.5]{Gou-asip} to conclude that (H2) holds.~\end{proof}

Given Lemma~\ref{lemma:wcontin}, we just need to recall previous results
on the Sina\u{\i} billiard map $(M,T,\mu)$ with infinite horizon
 and the cell-change function $\kappa:M\to\Z^d$, where (I1) and (I2c) were shown to hold in
  the anisotropic Banach spaces
  constructed for the billiard map in~\cite{DZ11, DZ13, DZ14}, and further studied in \cite{BBT}.
We recall that the probability measure preserving map $(\Omega,T,\mu)$
  is replaced by $(M,T,\mu)$ and
  the truncated observable $W_l^{R_l}$ is defined with $\kappa$ instead of $v$.

\textbf{Condition (I1) in Lemma~\ref{lemma:wcontin}} is stated for example in \cite[Theorem 2.5]{DZ11}, \cite[Theorem 2.2]{DZ13}, or \cite[Theorem 2.4]{DZ14}.
It is recalled also in \cite[Section 5.2]{BBT}.

\textbf{Condition (I2c) in Lemma~\ref{lemma:wcontin}} is verified in a stronger form
in~\cite[Section 5.1]{BBT}. Literally, \cite{BBT} considers only the untruncated case $R=\infty$, in particular, \cite[Formula (29)]{BBT} implies that there exists some exponent $\nu>0$ and a constant $C>0$ such that
\[
\|L_{t,\infty}-L \|_{\cB\to\cB} \le C|t|^\nu.
\]
In fact, \cite[Formula (29)]{BBT}  also investigates how the modulus of continuity depends on the size of the scatterers, but this can be taken into account as a uniform constant in the present case.
To see that the estimate extends to finite truncation levels $R<\infty$  with the same $\nu$ and $C$, note that the argument in~\cite[Section 5.1]{BBT} relies on the growth lemma type estimates \cite[Formula (19)]{BBT} and
\cite[Formula (25)]{BBT}, along with the bound $\left|(e^{it\kappa}-1)|_{W_i} \right|\ll |t|^{\nu} |\kappa|^{\nu}|_{W_i}$, where $W_i$ is some unstable curve. Now the effect of the truncation is that for unstable curves $W_i$ on which $|\kappa|$ exceeds $R$ the difference  $e^{it\kappa^R}-1$ simply vanishes, thus these curves make no contribution, and the total sum is definitely less than in the untruncated case $R=\infty$.

\section{Proof of Proposition~\ref{thm-main}}
\label{sec:main}

The crucial ingredient for the proof of Proposition~\ref{thm-main} is Proposition~\ref{prop:cr}. To state it, let us introduce the following notation:
\begin{equation}\label{eq:whatishj}
\text{Given } j\ge 1, \text{ let }\hj=2^{\lfloor\log_2 j\rfloor},
\end{equation}
that is, $\hj$ denotes the smallest power of $2$ that does not exceed $j$.

\begin{prop}\label{prop:cr} Assume (H1)--(H3). Let $(c_n)_{n\ge 1}$ be defined as in (HL0).
Given $j\ge 1$, let $\hj$ be as in \eqref{eq:whatishj}, and let $W_j^{c_{\hj}}$ be as in~\eqref{eq:wjdn} with $R=c_{\hj}$.

There exists a sequence of independent Gaussian random vectors $Z_j$ distributed as $\mathcal N(0,I)$  so that almost surely, as $n\to\infty$,
\[
 \left|\sum_{j=1}^n(W_j^{c_{\hj}}-\sigma_jZ_j)\right|=o(c_{n}),
\]
where $\sigma_j$ is a positive definite symmetric matrix so that $\sigma_j^2=\ell^*(c_{\hj})\,\Sigma^2(1+o(1))$(with $\Sigma$ and $\ell^*$ as in (H3)(iii)).
\end{prop}
The proof of Proposition~\ref{prop:cr} is provided in Section~\ref{sec:mainprop}.
The remainder of the proof of Proposition~\ref{thm-main} comes down to recalling a few results from~\cite{Ein07}, which do not require any form of independence.

\begin{lemma}{~\cite[Proof of Theorem 1, Step 2]{Ein07}~\cite[Item (iii) of Section 3.2]{Ein09}}\label{prop:indapp}
  Let $\Gamma_n=\ell^*(c_n) \Sigma^2$.
  Let $Z_j$ the independent Gaussians defined in Proposition~\ref{prop:cr}.
 Then, almost surely, as $n\to\infty$,
\[
\sum_{j=1}^n(\Gamma_j-\sigma_j) Z_j=o(c_{n}).
\]
\end{lemma}

We can now complete

\begin{pfof}{Proposition~\ref{thm-main}}
By Proposition~\ref{prop:cr}
and Lemma~\ref{prop:indapp},
\(
 \sum_{j=1}^n W_j^{c_{\hj}}-\sum_{j=1}^n\Gamma_j Z_j=o(c_{n}).
\)
Note that
\begin{align*}
 \sum_{j=1}^{\infty} \mathbb{P} (v_j\ne W_j^{c_{\hj}})
=\sum_{j=1}^{\infty} \mathbb{P} (v_j>c_{\hj}) =  \sum_{j=1}^{\infty} \frac{\ell(c_{\hj})}{c_{\hj}^2}
&\ll\sum_{n=1}^{\infty} \frac{2^n\ell(c_{2^n}/2^n)}{c_{2^n}^2}.
\end{align*}
Recall that $\ell(n)=o(\tilde\ell(n))=o(L(n))$. Thus,
\[
 \sum_{n=1}^{\infty} \frac{2^n\ell(c_{2^n}/2^n)}{c_{2^n}^2}\ll \sum_{n=1}^{\infty} \frac{LL(n)}{n\ell_1(2^n)}<\infty,
\]
where in the last inequality we have used (HL0)(iii).

By the Borel-Cantelli lemma, $\sum_{j=1}^n v_j-\sum_{j=1}^n W_j^{c_{\hj}}=o(c_{n})$.
Hence,
\begin{align}\label{firstc}
\sum_{j=1}^n v_j-\sum_{j=1}^n\Gamma_j Z_j=o(c_{n}).
\end{align}

By~\cite[Equation (3.16)]{Ein09} (see also~\cite[Proof of Theorem 1, Steps 3 and 4]{Ein07}),
\(
 \max_{1\le k\le n}\left|\sum_{j=1}^k(\Gamma_n-\Gamma_j) Z_j\right|=o(c_{n}).
\)
This together with~\eqref{firstc} implies that
\[
 \max_{1\le k\le n}\left|\sum_{j=1}^k v_j-\Gamma_n\sum_{j=1}^k Z_j \right|=o(c_{n}).
\]~\end{pfof}

\section{Proof of Proposition~\ref{prop:cr}}
\label{sec:mainprop}

Proceeding as in~\cite[Section 5]{Gou-asip}, we divide $\N$ into  intervals $[2^n, 2^{n+1})$ and  further divide each such interval into intervals of `good' length with gaps between them.

Fix $\beta\in (0,1)$ and $\eps_1\in (0, 1-\beta)$.
Decompose each interval $[2^n, 2^{(n+1)})$ into a union of
$F(n)= 2^{\lfloor \beta n\rfloor}$ intervals $I_{n,j}, 0\le j < F(n)$ of equal length
and place $F(n)$ gaps $J_{n,j}, 0\le j < F(n)$ between them.
That is, $[2^n, 2^{n+1})= J_{n,0}\cup I_{n,0}\cup J_{n,1}\cup\ldots I_{n,F(n)-1}$.

We choose the lengths of the gaps in the same way as in~\cite[Section 5]{Gou-asip}. In particular, for $1\le j\le F(n)$, write $j=\sum_{k=0}^{2^{\lfloor n\beta\rfloor}}\alpha_k(j) 2^k$, $\alpha_k(j)\in\{0,1\}$
and let $r$ be the smallest number so that $\alpha_r(j)\ne 0$. In this notation,

\begin{align}\label{gaps}
 |J_{n,j}|=2^{\lfloor \eps_1 n\rfloor }2^r
 \text{ for } j\ne 0\text{ and }
 |J_{n,0}|= 2^{\lfloor \eps_1 n\rfloor }2^{\lfloor \beta n \rfloor }.
\end{align}

As such, the lengths of the gaps add up to $ 2^{\lfloor\eps_1 n\rfloor}2^{\lfloor \beta n\rfloor-1}(\lfloor \beta n\rfloor+2)$ and one can take
\begin{align}\label{eq:Inj}
 |I_{n,j}|=2^{n-[n\beta]}-(\lfloor\beta n\rfloor+2) 2^{\lfloor\eps_1 n\rfloor-1}.
\end{align}

The first proposition below uses the gaps (via the fast decorrelation expressed in (H2)) to show the sums $\sum_{l\in I_{n,j}} W_l^{c_{2^n}}$ can be regarded independent for distinct intervals $I_{(n,j)}$
(at the price of an almost surely $o(c_n)$ error). It is analogous to~\cite[Proposition 5.1]{Gou-asip} under  assumption (H2).
 For its statement we need the lexicographical order
 $\prec$ (as used in \cite[Section 5]{Gou-asip}): we write $(n', j')\prec (n,j)$ if the interval $I_{n',j'}$ is to the left of $I_{n,j}$
 for $0\le j'  < F(n')$ and $0\le j < F(n)$.

\begin{lemma}\label{prop:i.i.d.app} Assume (H1)--(H3).
Let $X_{n, j}=\sum_{l\in I_{n,j}} W_l^{c_{2^n}}$.

There exists a coupling between $(X_{n,j})_{n\in\N,\, 0\le j<F(n)}$  and a sequence of independent random variables
$(\hat X_{n,j})_{n\in\N,\, 0\le j<F(n)}$  with $\hat X_{n,j}$ distributed as $ X_{n,j}$,
such that, almost surely as $(n,j)\to\infty$,
  \[
   \left|\sum_{(n', j')\prec (n,j)}(X_{n',j'}-\hat X_{n',j'})\right|=
   o\left(c_{2^n}\right).
  \]
\end{lemma}
The proof of this lemma is provided in Subsection~\ref{sec:prooft}.

The next lemma is an analogue of~\cite[Section 5.2]{Gou-asip} (coupling with Gaussians) under the current assumptions.

\begin{lemma}\label{lemma:Gaus} Assume (H1)--(H3), (HL0) and (HL1).
Let $(\hat X_{n,j})_{n\in\N,\, 0\le j<F(n)}$  be the independent random variables obtained in Lemma~\ref{prop:i.i.d.app}.

There exits a coupling between $(\hat X_{n,j})_{n\in\N,\, 0\le j<F(n)}$ and a sequence of independent Gaussian random variables $(\tZ_{n,j})_{n\in\N,\, 0\le j<F(n)}$ with covariance matrices
\[
cov (\tZ_{n,j})=cov (\hat X_{n,j})= \ell^*(c_{2^n}) \cdot |I_{n,j}| \cdot \Sigma^2 (1+o(1))
\] so that
almost surely as $(n,j)\to\infty$,
  \[
   \left|\sum_{(n', j')\prec (n,j)}(\hat X_{n',j'}-\tZ_{n',j'})\right|=
   o(c_{2^n}).
  \]
  \end{lemma}
The proof of this lemma is provided in Subsection~\ref{sec:Gaus}.

As a consequence of the two previous lemmas, we obtain
\begin{cor}\label{cor:f} Let $X_{n,j}$ and $\tZ_{n,j}$ as defined in Lemmas~\ref{prop:i.i.d.app}and in~\ref{lemma:Gaus}, respectively.
 Then
 \[
   \left|\sum_{(n', j')\prec (n,j)}(X_{n',j'}-\tZ_{n',j'})\right|=
   o\left(c_{2^n}\right).
  \]
\end{cor}

We turn to maxima of sums and gaps.

\begin{lemma}\label{lemma:max} Assume (H3), (HL0) and (HL1).
Let $i_{n,j}$ be the smallest element of $I_{n,j}$.
Then the following holds almost surely as $(n,j)\to\infty$
 \[
  \max_{r<|I_{n,j}|}\left|\sum_{l=i_{n,j}}^{i_{n,j}+r}W_l^{c_{2^n}}\right|=o\left(c_{2^n}\right)
   \]and a similar statement holds for $\tZ_{n',j'}$.\end{lemma}
The proof of Lemma~\ref{lemma:max} is given in Subsection~\ref{sec:maxpr}.

The next lemma shows that the gaps $J_{n,j}$ (see~\eqref{gaps}) can be neglected.

\begin{lemma}\label{lemma:gaps} Let $c_n$ as in (HL0) and assume (H2) and (H3).
Let $\mathcal J=\cup_{n,j} J_{n,j}$. Then, for any $\eps>0$, almost surely, as $k\to\infty$,
\[
 \left|\sum_{l<k, l\in\mathcal J} W_l^{c_k}\right|=o\left(k^{\beta/2+2\eps+\eps_1}\right).
\]
A similar statement holds for $\tZ_{n',j'}$.
 \end{lemma}
 Lemma~\ref{lemma:gaps} is proved in Subsection~\ref{sec:gaps}
by controlling the $L^2$ norms of the truncated variables in the same manner as in~\cite[Proof of Lemma 5.9]{Gou-asip}.

We can complete

\begin{pfof}{Proposition~\ref{prop:cr}}
First note that, as for any $(n,j)$ the Gaussian variable $\tZ_{n,j}$ has covariance $|I_{n,j}|\ell^*(c_{2^n}) \Sigma^2 (1+o(1))$\footnote{Note that for $l\in I_{n,j}$ we have $\hl=2^n$.}, it can be regarded as a sum of independent
Gaussians, $\tZ_{n,j}=\sum_{l\in I_{n,j}} \tZ_l$, where $\tZ_l$ is distributed as $\cN(0,\sigma_l^2)$, with $\sigma_l^2=\ell^*(c_{\hl}) \Sigma^2 (1+o(1))$ as in the statement of the Proposition.

Now let $\mathcal I=\cup_{n,j} I_{n,j}$ be the union of the main blocks. By Corollary~\ref{cor:f} and Lemma~\ref{lemma:max} , $\left|\sum\limits_{l\in\mathcal I, l<k}
\left(W_l^{c_{\hl}}-\tZ_{l}\right)\right|=o(c_{\hk})=o(c_k)$, almost surely, as $k\to\infty$. By Lemma~\ref{lemma:gaps}, $\left|\sum\limits_{l\in\mathcal J, l<k}
W_l^{c_{\hl}}\right|=o\left(k^{\beta/2+2\eps+\eps_1}\right)=o(c_{k})$; in the last equalities
we have used that $\beta<1$ and that $\eps,\eps_1$ can be arbitrarily small
and also that $k\in\mathcal J$.
Thus, $\left|\sum_{j=1}^k W_j^{c_{\hj}}-\tZ_{j}\right|=
 o(c_{k})$, almost surely, as $k\to\infty$.

By definition, $\tZ_{j}=\sigma_j Z_j$, where as in the statement of Proposition~\ref{prop:cr}, $\sigma_j^2= \ell^*(c_{\hj}) \Sigma^2 (1+o(1))$, and $Z_j$ is a sequence of Gaussian random
vectors distributed as $\cN(0,I)$.
Thus, almost surely, as $k\to\infty$,
\begin{align*}\label{maintr}
 \left|\sum_{j=1}^k (W_j^{c_{\hj}}-\tZ_{j})\right|=\left|\sum_{j=1}^k (W_j^{c_{\hj}}-\sigma_j\,Z_{j})\right|=
 o(c_{k}).
\end{align*}
\end{pfof}

\section{Proofs of the lemmas used in the proof of
Proposition~\ref{thm-main}}
\label{sec:mainlem}

\subsection{Proof of Lemma~\ref{prop:i.i.d.app}}
\label{sec:prooft}

The proof of this lemma is analogous to~\cite[Proof of Proposition 5.1]{Gou-asip}.
We recall the main elements partly for completeness, partly for use in some proofs below.

The proof exploits assumption (H2) and the Strassen Dudley theorem (see, for instance,~\cite[Theorem 6.9]{Billingsley}), which we now recall.
The Prokhorov metric $\pi(Y,Z)$ between two random variables $Y$ and $Z$ taking values in some metric space $M$ is defined as
 \[
 \inf \{\eps>0 \,|\, \pP(Y\in A) < \pP(Z\in A^{\eps}) + \eps \text{ and } \pP(Z\in A) < \pP(Y\in A^{\eps}) + \eps, \forall  A\in \mathcal{B}(M) \},
 \]
where $\mathcal{B}(M)$ denotes the collection of Borel sets of $M$, and $A^{\eps}$ denotes the $\eps$-neighborhood of $A$.
Specifically, for real-valued random variables $Y$ and $Z$  with  distribution functions $F_{Y}$ and $F_{Z}$ this reduces to
 \[
 \inf \{\eps>0 \,|\, F_Y (x-\eps)-\eps \le F_Z (x) \le  F_Y (x+\eps)+\eps, \forall x\in\R \}.
 \]
 In particular, if $|F_Y (x)-F_Z (x)|\le \eps$ for every $x\in\R$, then $\pi(Y, Z)\le \eps$. In the statement below,
 $\R^{dM}$ is considered with the norm $|(x_1,\ldots, x_M)|=\sup_{1\le i\le M}|x_i|$, where $|x|$ stands for the Euclidian norm of $x\in\R^d$.

 \begin{thm}{~\cite[Theorem 3.4, Lemma 3.5]{Gou-asip}}
 \label{thm-StrD}
 Let $Y, Z$ be two random variables taking values in $\R^{d M}$ with probability distributions
 $F_Y, F_Z$ and characteristic functions $\Psi_Y,\Psi_Z$. Let $\pi(Y,Z)$ denote the Prokhorov distance
 between these distributions.

(a) If $\pi(Y, Z)< c$ then there exists a coupling between $Y$ and $Z$ such that $\mathbb{P}(|Y-Z|>c)<c$.

(b) There exists $C_d>0$  so that for any $T^*>0$,
 \begin{align*}
    \pi(Y,Z)\le
    \sum_{j=1}^M P_Y(|y_j|>T^*)
  +(C_d T^*)^{(dM)/2}\left(\int_{\R^{d M}}|\Psi_Y(t)-\Psi_Z(t)|^2\, dt\right)^{1/2}.
  \end{align*}

\end{thm}

 We can now recall the main steps of~\cite[Proof of Proposition 5.1]{Gou-asip}. We emphasise that these steps require (H2), (H3), Theorem~\ref{thm-StrD} and that the observable of interest (in our set
 up $W_{l}^{c_{2^n}}$) is in $L^1(\mu)$.
\\[2mm]
\begin{proofof}{Lemma~\ref{prop:i.i.d.app}}
Recall $X_{n,j}=\sum_{l\in I_{n,j}} W_l^{c_{2^n}}$
Let
$X_n=(X_{n,j})_{0\le j < F(n)}$
    be vectors in $\R^{d F(n)}$.
To exploit (H2), we shall proceed as in~\cite{Gou-asip}
 working with variants of $X_{n,j}$ whose characteristic functions will be supported on
 a neighborhood of $0$. More precisely, given $\eps_0$ as in (H2),~\cite[Proposition 3.8]{Gou-asip} ensures that there exists
 a symmetric random variable $V$ in $\R^d$, $V\in L^q$, for every $q\ge 2$, whose characteristic function is supported in $\{|t|\le\eps_0\}$.
 Let $V_{n,j}$ be independent copies of $V$, independent of everything else. Set
 $X_{n,j}^*=X_{n,j}+V_{n,j}$
 and let $X_n^*=(X_{n,0}^*,X_{n,1}^*,\ldots,  X_{n,F(n)-1}^*)$ be vectors in $\R^{dF(n)}$.

Let $Q_n$ be a random variable distributed like $X_n$ and  consider the $\R^{d D(n)}$ random vectors
$(X_1, X_2,\ldots, X_{n-1}, X_n)$
 and $(X_1, X_2,\ldots, X_{n-1}, Q_n)$, where
 $D(n)=\sum_{j=1}^n F(j)\ll 2^{\beta n}$. Note that $J_{n,0}$ is a gap
 of length $k\approx 2^{\eps_1 n+\beta n}$ between between $(X_j)_{j<n}$
and $X_n$. Let $\phi$ and $\gamma$ be the characteristic
 functions of $(X_1, X_2,\ldots, X_{n-1}, X_n)$ and $(X_1, X_2,\ldots, X_{n-1}, Q_n)$. Reasoning as in the first part of the proof of~\cite[Lemma 5.2]{Gou-asip},
 we have under the current (H2) that
 for all $|t_{m,j}|\le\eps_0$ and for $n$ large enough,
 \[
  |\phi-\gamma|\ll e^{-c2^{\eps_1 n+\beta n}}.
 \]
Set $Q_n^*=Q_n+V_n$, where $V_n$ is an independent copy of $V$, independent of everything else.
Consider the $\R^{d D(n)}$ random vectors
$(X_1^*, X_2^*,\ldots, X_{n-1}^*, X_n^*)$
 and $(X_1^*, X_2^*,\ldots, X_{n-1}^*, Q_n^*)$ with characteristic
 functions
 $\phi^*$ and $\gamma^*$, which  are obtained by multiplying $\phi,\gamma$
 by the characteristic function of $V$. Thus, $|\phi^*-\gamma^*|\ll e^{-c2^{\eps_1 n+\beta n}}$.

As in the last paragraph of~\cite[Proof of Lemma 5.2]{Gou-asip},
the distance between the distributions requires only that for all $n$, $\|W_{l}^{c_{2^n}}\|_{L^1(\mu)}\le C$,
 for some $C>0$ independent of $n$ and $l$; this is true
 because $v\in L^1(\mu)$. We apply Theorem~\ref{thm-StrD}(b) with $T^*=e^{2^{\eps_1 n/2}}$ and obtain, by the Markov inequality, that the terms associated
 to the tails of the distributions are bounded by $\ll 2^n e^{-2^{\eps_1 n/2}}$.

 Putting the previous two facts together and using Theorem~\ref{thm-StrD},
 \begin{align}\label{eq:g1}
 \pi\left((X_1^*, X_2^*,\ldots, X_{n-1}^*, X_n^*), (X_1^*, X_2^*,\ldots, X_{n-1}^*, Q_n^*)\right)\ll 4^{-n}.
\end{align}

Using \eqref{eq:g1} and proceeding as in~\cite[Proof of Corollary 5.3]{Gou-asip}, we obtain that there exists a coupling
between $(X_1^*, X_2^*,\ldots)$ and $(R_1^*, R_2^*,\ldots)$,
where the variables $R_n^*=(R_{n,j}^*)_{j<F(n)}$ are distributed as the $X_n^*$, but independent of each other, so that
\begin{align}\label{eq:g2}
 \mu\left(|X_{n,j}^*-R_{n,j}^*|\ge C4^{-n}\right)\le C4^{-n}.
\end{align}

The next step is then to show -- based on (H2), with an argument analogous to~\cite[Proof of Lemma 5.4]{Gou-asip} -- that
\begin{align}\label{eq:g3}
 \pi\left( (R_{n,j}^*)_{0\le j<F(n)}, (\hat X_{n,j}^*)_{0\le j<F(n)}\right)\le C4^{-n},
\end{align}
where $\hat X_{n,j}^*= \hat X_{n,j}+ V_{n,j}$ with
$(\hat X_{n,j})_{n\in\N, 0\le j<F(n)}$ a sequence of independent random variables distributed like $X_{n,j}$. We recall the main ingredients of the argument in~\cite[Proof of Lemma 5.4]{Gou-asip} needed to obtain~\eqref{eq:g3},
partly for completeness, partly for use in the proof of Lemma~\ref{lemma:gaps} below.
The main idea of~\cite[Proof of Lemma 5.4]{Gou-asip} is to make the variables $(R_{n,j}^*)_{0\le j<F(n)/2}$ independent of $(R_{n,j}^*)_{F(n)/2<j<F(n)}$ by using a large gap $J_{n,F(n)/2}$ and proceed in each remaining half using the gap in the middle of this half and repeating the procedure. For $0\le i<\lfloor n\beta\rfloor$ define $\tilde Y_{n,j}^i$ so that for $0\le k<2^{\lfloor n\beta\rfloor-1}$, the random variable
$\mathcal{\tilde Y}_{n,k}^i=( \tilde Y_{n,j}^i)_{k2^i\le j<(k+1)2^i}$ is distributed like $(X_{n,j}^*)_{k2^i\le j<(k+1)2^i}$ and so that $\mathcal{\tilde Y}_{n,k}^i$ is independent
of $\mathcal{\tilde Y}_{n,k'}^i$ if $k\ne k'$. The argument of~\cite[Proof of Lemma 5.4]{Gou-asip} simply goes through provided that the distance between the characteristic function
of the random variables $\mathcal{\tilde Y}_{n,k}^i$ and
$(\mathcal{\tilde Y}_{n,2k}^{i-1}, \mathcal{\tilde Y}_{n,2k+1}^{i-1})$ is $O(e^{-c2^{n\eps_1+i}})$. Since $\mathcal{\tilde Y}_{n,k}^i$ consists of $2^i$ sums of variables
$W_l^{c_{2^n}}$ along blocks of length $2^{n-\lfloor n\beta\rfloor}$ with a gap
of size $\approx 2^{n\eps_1+i}$ in the middle, (H2) applies and, the required result $O(e^{-c2^{n\eps_1+i}})$ follows. Finally, let $\hat X_{n,j}^*=\tilde Y_{n,j}^0$.

Using Theorem~\ref{thm-StrD}, equations~\eqref{eq:g1} and~\eqref{eq:g3}, and Borel-Cantelli lemma (after summing over $j$ and $n$), the following holds almost surely:
\[
   \left|\sum_{(n', j')\prec (n,j)}(X_{n',j'}^*-\hat X_{n',j'}^*)\right|=
   o(2^{n/2})=o(c_{2^n}).
  \]
To obtain the result for $X_{n,j}$, one needs to argue that the contribution of  $V_{n,j}$ is negligible. As in~\cite[Proof of Proposition 5.1]{Gou-asip}, this can be done using the law of the iterated logarithms
for the centered, independent random variables $V_{n,j}\in L^2(\mu)$:
\begin{equation*}
 \left|\sum_{(n', j')\prec (n,j)}V_{n',j'}\right|\le
 C N_{n,j}^{1/2}\cdot \sqrt{\log\log(N_{n,j})},\text{ almost surely},
\end{equation*}
where $N_{n,j}$ is the cardinality of the set
$\{(n', j'): (n', j')\prec (n,j)\}$.
Since $N_{n,j} \le \sum_{n'=1}^n \sum_{j'=0}^{F(n')-1} 1 \le C 2^{\beta n}$,
\[
 \left|\sum_{(n', j')\prec (n,j)}V_{n',j'}\right|=o(2^{(\beta n+\eps n)/2}),
\]
for any $\eps>0$. Since $\beta<1$ and $\eps$ is arbitrary, $o(2^{(\beta n+\eps n)/2})=o(2^{n/2})$, we obtain
\begin{align}\label{eq:Aaa}
 \left|\sum_{(n', j')\prec (n,j)}(X_{n',j'}-\hat X_{n',j'})\right|=
   o(2^{n/2})=o(c_{2^n}),\text{ almost surely},
  \end{align}
as required.~\end{proofof}

 \subsection{Proof of Lemma~\ref{lemma:Gaus}}
\label{sec:Gaus}

Let $(\hat X_{n,j})_{n\in\N,\, 0\le j<F(n)}$  be the independent random variables obtained in Lemma~\ref{prop:i.i.d.app}.

In this section we exploit assumption (HL1) with $d_n=\sqrt{n L(n) \ell_1(n)^{-99}}$ as in~\eqref{eq:dn} and $r=i_{n,j}$.
More precisely, fix $(n,j)$ and for $l=0,\dots, |I_{n,j}|$ let
\[
\tW= v\cdot \mathbf{1}_{d_{2^n}\le|v|\le c_{2^n}}, \qquad \tW_l=W_{l+i_{n,j}}^{c_{2^n}} -W_{l+i_{n,j}}^{d_{2^n}}=\tW\circ T^{l+i_{n,j}}.
\]

Note that
\begin{equation}\label{eq:d_n_c_n}
\ell^*(c_{2^n}) \asymp \ell^*(d_{2^n}) \asymp n \quad \text{and} \quad \frac{c_n}{d_n} \asymp (\ell_1 (n))^{50}.
\end{equation}
Let $Y_{n,j}=\sum_{l\in I_{n,j}} W_l^{d_{2^n}}$, and let $(\hat{Y}_{n,j})_{n\in\N, 0\le j<F(n)}$ be a sequence of independent r.v.s distributed as $(Y_{n,j})_{n\in\N, 0\le j<F(n)}$.

The first result below says that the additional truncation in terms of $d_{2^n}$ causes an error which is almost surely $o(c_{2^n})$.

\begin{lemma}\label{lemma:d_n_c_n}
For any $\varepsilon>0$ we have
\[
\sum_{n=1}^{\infty}  \mu \left( \left|  \sum_{j=0}^{F(n)-1} (\hat{X}_{n,j} - \hat{Y}_{n,j}) \right|>\varepsilon c_{2^n} \right) <\infty.
\]
\end{lemma}

\begin{proof}
We will use Chebyshev's inequality, hence need a bound on the second moment of
\[
\sum_{j=0}^{F(n)-1} (\hat{X}_{n,j} - \hat{Y}_{n,j})= \sum_{j=0}^{F(n)-1} \sum_{l\in I_{n,j}} (W_l^{c_{2^n}} -W_l^{d_{2^n}}).
\]
Using (HL1) and~\eqref{eq:d_n_c_n},
\begin{align}\label{eq:tilde_w_second}
\nonumber\E_\mu \left(\left|\sum_{j=0}^{F(n)-1} (\hat{X}_{n,j} - \hat{Y}_{n,j})\right|^2\right) &\ll  F(n) \cdot \E_\mu\left(\left|\sum_{l\in I_{n,j}}(\tW_l)\right|^2\right) \\                                                                                                                                                                                          &\ll F(n) \cdot |I_{n,j}| \cdot L\left( \frac{c_{2^n}}{d_{2^n}}\right)\ll 2^n \cdot LL(n).
\end{align}
Hence, Chebyshev's inequality implies
\begin{align*}
\mu \left( \left|\sum_{j=0}^{F(n)-1}  \hat{X}_{n,j} - \hat{Y}_{n,j} \right|>\varepsilon c_{2^n} \right) \ll& \frac{2^n \cdot LL(n)}{ (c_{2^n})^2}
\ll\frac{LL(n)}{n \ell_1(2^n)}
\end{align*}
which is summable in $n$, by assumption (HL0)(iii).
\end{proof}

It remains to couple the independent variables $\hat{Y}_{n,j}$ truncated at the lower level $d_{2^n}$ to appropriate Gaussians, as stated in the next Lemma, the analogue of~\cite[Lemma 5.6]{Gou-asip}
in the present set up.

\begin{lemma}\label{lemma:gausshatynj}
 Let $n\in\N$. There exists a coupling between $(\hat{Y}_{n,0},\hat{Y}_{n,1},\ldots, \hat{Y}_{n,F(n)-1})$ and $(S_{n,0},S_{n,1},\ldots, S_{n,F(n)-1})$,
 where $(S_{n,j})_{n\in\N,0\le j<F(n)}$ is a sequence of independent, centered  Gaussian vectors with $cov S_{n,j}=cov \hat{Y}_{n,j}$ and so that for any $\eps>0$,
 \[
  \sum_n\mu\left(\max_{1\le i\le F(n)}\left|\sum_{j=0}^{i-1}Y_{n,j}-S_{n,j}\right|\ge \frac{c_{2^n}}{L(n)}\right)<\infty.
 \]
\end{lemma}

\begin{proof}
 We apply~\cite[Corollary 3]{zait} as recalled
 in~\cite[Proposition 5.5]{Gou-asip} to $\hat{Y}_{n,j}$.
 We briefly recall the statement. Let $Y_0,\ldots, Y_{b-1}$ be independent centered $\R^d$-valued random variables. Let $p\ge 2$ and
 $M=\left(\sum_{j=0}^{b-1}\E|Y_j|^p\right)^{1/p}$. Assume that there exists a sequence $0=m_0<m_1<\ldots m_s=b$ so that given $\Xi_k= Y_{m_k}+\ldots +Y_{m_{k+1}-1}$, $B_k=cov\Xi_k$,
 the following holds for all $w\in\R^d$, for all $0\le k\le s$ and for some $C>1$:
 \begin{align}\label{eq:MG}
  100 M^2|w|^2\le \langle B_k w,w\rangle\le 100 CM^2 |w|^2.
 \end{align}
Then there exits a coupling between $(Y_0,\ldots Y_{b-1})$ and a sequence of independent Gaussian random vectors $(S_0,\ldots S_{b-1})$  with $cov S_j=cov Y_j$
and so that
\begin{align}\label{eq:conMG}
 \mu\left(\max_{1\le i\le b}\left|\sum_{j=0}^{i-1}Y_{j}-S_{j}\right|\ge Mz\right)<
 c_0z^{-p}+e^{-C_0z},
\end{align}
for all $z\ge C_0\log s$ with $C_0$ a constant that depends only on $C$, dimension $d$ and exponent $p$.

We will apply this statement with $b=F(n)$, $Y_j=\hat{Y}_{n,j}$, $p=4$ and $z=n^{1/4} \cdot L(n)$. Let us check that the conditions hold.
By (H3)(ii) we have (recall $|I_{n,j}|\ll 2^{(1-\beta) n}$ with $\beta<1$ while $2^{n/2}= o(d_{2^n})$, so the required relation of number of terms and truncation level, $|I_{n,j}|=o((d_{2^n})^2)$,  holds)
\[
\E(|\hat{Y}_{n,j}|^4) \ll |I_{n,j}| (d_{2^n})^2,
\]
and thus
\begin{align}
M=\left(\sum_{j=0}^{F(n)-1} \E(|\hat{Y}_{n,j}|^4) \right)^{1/4} \ll & \left( F(n) 2^{(1-\beta) n} d_{2^n}^2 \right)^{1/4} \ll 2^{n/4} d_{2^n}^{1/2} \nonumber\\
\ll & 2^{n/2}\left( L(2^n) \ell_1(2^n)^{-80} \right)^{1/4}\ll 2^{n/2} n^{1/4} (L(n))^{-20}. \label{eq:Mwith4}
\end{align}
This implies
\[
M^2\ll 2^n \sqrt{n},
\]
while, by (H3) (ii)
\[
\sum_{j=0}^{F(n)-1} cov \hat{Y}_{n,j} \asymp \sum_{j=0}^{F(n)-1} |I_{n,j}| \ell^*(d_{2^n}) \gg 2^n n
\]
which is larger in magnitude than $M^2$. On the other hand, recalling that $\Sigma$ is non degenerate
(by (H1)), there must exist some constant $C_1$ so that for any large $n$, for any  $0\le m<m'<F(n)$ for any vector $w$ in $\R^d$,
\[
 C_1^{-1}(m'-m)2^{(1-\beta)n} |w|^2 \le
 \Big\langle \sum_{j=m}^{m'-1} cov Y_{n,j} w, w\Big\rangle
 \le C_1(m'-m)2^{(1-\beta)n}|w|^2
\]
and each individual term
is bounded by
\[
 |cov Y_{n,j}|\, |w|^2\le \E_\mu(|Y_{n,j}|^{2})|w|^2\le (\E_\mu(|Y_{n,j}|^{4}))^{1/2}|w|^2\le M^2|w|^2.
\]
So, one can group $Y_{n,j}$ into consecutive blocks so that~\eqref{eq:MG} holds
with some constant $C$. Note also that, for the number of blocks, $s=\sqrt{n} \ell_2(n)$
with some slowly varying function $\ell_2$, hence the condition $z\gg \log s$ holds with the above specified $z=n^{1/4} L(n)$.

Now, using \eqref{eq:Mwith4}, we have
\[
z\cdot M =n^{1/4} L(n) 2^{n/2} n^{1/4} (\ell_1(2^n))^{-20} \ll \sqrt{2^n n \ell_1(2^n)} (\ell_1(2^n))^{-1} \ll \frac{c_{2^n}}{L(n)}
\]
where we have used that $L(n)\ll \ell_1(2^n)$ by (HL0)(ii) (see also (HL0)(iii)). The conclusion follows by \eqref{eq:conMG} as
\begin{align*}
\mu\left(\max_{1\le i\le F(n)}\left|\sum_{j=0}^{i-1}Y_{n,j}-S_{n,j}\right|\ge \frac{c_{2^n}}{L(n)}\right) \, < \,&
 \mu\left(\max_{1\le i\le b}\left|\sum_{j=0}^{i-1}Y_{j}-S_{j}\right|\ge Mz\right)\\
 <\,& c_0z^{-4}+e^{-C_0z} \, < \, c_0 (n L(n)^4)^{-1} + e^{-C_0n^{1/4}}
\end{align*}
which is summable in $n$.
\end{proof}

Lemmas~\ref{lemma:d_n_c_n} and~\ref{lemma:gausshatynj} imply that there exist a coupling such that
\[
\max_{1\le i\le F(n)}\left|\sum_{j=0}^{i-1}\hat{X}_{n,j}-S_{n,j}\right|=o(c_{2^n})
\]
almost surely, as $n\to \infty$. As
\[
\sum_{m=0}^n c_{2^m} \le \left(\sum_{m=0}^n 2^{m/2}\right) \sqrt{\ell^*(2^n) \ell_1(2^n)} \ll c_{2^n},
\]
it follows that
\[
 \sum_{(n', j')\prec (n,j)}\hat{X}_{n',j'}-S_{n',j'}=o(c_{2^n}),
\]
almost surely, as $(n,j)\to \infty$.

We can now complete

\begin{pfof}{~Lemma~\ref{lemma:Gaus}} Let $S_{n,j}$ be independent, centered  Gaussian vectors  obtained in Lemma~\ref{lemma:gausshatynj}.
We still need to construct a sequence of independent Gaussian random vectors $\tZ_{n,j}$ with $cov (\tZ_{n,j})= \ell^*(c_{2^n}) \cdot |I_{n,j}| \cdot \Sigma^2$
and a coupling between $S_{n,j}$ and $\tZ_{n,j}$. We may rely on \cite[Lemma 5.7]{Gou-asip}, which provides such a construction,  such that
\[
 \sum_{(n', j')\prec (n,j)}\left(\frac{S_{n',j'}}{\sqrt{\ell^*(c_{2^n})}}-\frac{\tZ_{n',j'}}{\sqrt{\ell^*(c_{2^n})}}\right)=o(2^{(\beta+\eps_1)n/2}).
\]
The conclusion follows -- together with
Lemma~\ref{lemma:gausshatynj}-- as $c_n=o(\sqrt{\ell^*(c_{2^n})}\cdot 2^{(\beta+\eps_1)n/2})$, and $\beta+\eps_1<1$.~\end{pfof}

\subsection{Proof of Lemma~\ref{lemma:max}}
\label{sec:maxpr}

Throughout this section we use assumptions (H3)(ii), (HL2) and (HL3), along with the corresponding notation.
In particular, with $\bd_{2^n}=2^{n(1-\varsigma)/2}$ (as in \eqref{eq:dn}), recall the notation $\bW= v\cdot \mathbf{1}_{\bd_{2^n}\le|v|\le c_{2^n}}$, and let, with $r=i_{n,j}$,
$\bW_{l}=\bW\circ T^{l+r}$ for $l=0,\dots,|I_{n,j}|$.

First, we need to derive the following Lemma from assumptions (H3)(ii) and (HL2).

\begin{lemma}\label{lemma:maxd_n_c_n}
For any $\eps>0$
\begin{equation}\label{eq:maxd_n_c_n}
\sum_{n=1}^{\infty} \sum_{j=0}^{F(n)-1}  \mu\left(\max_{r<|I_{n,j}|}\left|\sum_{l=0}^{r}\bW_l\right|\ge
   \eps c_{2^n}\right)<\infty.
\end{equation}
\end{lemma}

\begin{proof}
Let us consider first the following weaker bound
\begin{equation}\label{eq:nomaxd_n_c_n}
\sum_{n=1}^{\infty} \sum_{j=0}^{F(n)-1}  \mu\left(\left|\sum_{l=0}^{|I_{n,j}|} \bW_l\right|\ge
   \eps c_{2^n} \right)<\infty.
\end{equation}

It is useful to compare \eqref{eq:nomaxd_n_c_n} with the bound in Lemma~\ref{lemma:d_n_c_n}. Here we want to bound the tail properties for random variable
that arises as a sum with a number of terms which is $|I_{n,j}|\ll 2^{n(1-\beta)}$. This is far less than the number of terms, $\asymp 2^n$ in Lemma~\ref{lemma:d_n_c_n}.
In particular, in the context of \eqref{eq:nomaxd_n_c_n}, (H3)(ii) applies (for $\varsigma$ sufficiently small).
We combine it with the Markov inequality:
By (H3)(ii),
\begin{equation}\label{eq:bar_w_fourth}
\E_\mu\left(\left|\sum_{l\in I_{n,j}} (W_l^{c_{2^n}} -W_l^{\bd_{2^n}}) \right|^4\right)\                                                                                                                                                                                           \ll \ |I_{n,j}| \cdot (c_{2^n})^2.
\end{equation}
Hence,
\begin{align*}
\sum_{j=0}^{F(n)-1}   \mu\left(\left|\sum_{l=0}^{|I_{n,j}|} \bW_l\right|\ge
   \eps c_{2^n} \right) \ll F(n) \cdot \frac{|I_{n,j}| \cdot (c_{2^n})^2}{(c_{2^n})^4}
\ll& \frac{F(n) \cdot |I_{n,1}|}{ (c_{2^n})^2}
\ll\frac{1}{n \ell_1(2^n)}
\end{align*}
which is summable in $n$, by (HL0)(iii). This completes the verification of \eqref{eq:nomaxd_n_c_n}.

To proceed, we exploit (HL2) with $r=i_{n,j}$ and $K=|I_{n,j}|$. Note that
$r$ and $K$ are fixed since  $n$ and $j$ are fixed. In concrete terms, for $k=1,\dots, K$, we apply (HL2) to
\[
\cS_k=\sum_{l=0}^{k} \bW_l;\qquad B_{k,\eps}=\{x\in M \,|\, |\cS_k|\ge \eps c_{2^n}\}
\]
and \(
N_{k,\eps}=B_{k,\eps}\setminus (\cup_{k'=0}^{k-1} B_{k',\eps}).
\)
By (HL2) \quad \(
\sum_{k=1}^K \mu(N_{k,\eps} \setminus B_{K,\eps/2})\ll \frac{K}{(c_{2^n})^2}.
\)

By the argument used in~\eqref{eq:tilde_w_second},
\(
\mu(B_{K,\eps/2}) \ll \frac{K}{(c_{2^n})^2}.
\)
Thus,
\[
\mu\left( \cup_{k=0}^K B_{k,\eps}
\right)\ll \frac{K}{(c_{2^n})^2}.
\]
The required summability for \eqref{eq:maxd_n_c_n} follows by the argument used in the proof of Lemma~\ref{lemma:d_n_c_n}.~\end{proof}

\begin{proofof}{Lemma~\ref{lemma:max}} By Lemma~\ref{lemma:maxd_n_c_n}, it remains to prove
\[
\sum_{n=1}^{\infty} \sum_{j=0}^{F(n)-1}  \mu\left(\max_{r<|I_{n,j}|}\left| \sum_{l=i_{n,j}}^{i_{n,j}+r} W_l^{\bd_{2^n}} \right|\ge
   \eps c_{2^n}\right)<\infty.
\]

Note that, by stationarity,
\[\mu\left(\max_{r<|I_{n,j}|}\left| \sum_{l=i_{n,j}}^{i_{n,j}+r} W_l^{\bd_{2^n}} \right|\ge
   \eps c_{2^n}\right)= \mu\left(\max_{r<|I_{n,j}|}\left| \sum_{l=0}^{r} W_l^{\bd_{2^n}} \right|\ge
   \eps c_{2^n}\right).\] To bound $\mu\left(\max_{r<|I_{n,j}|}\left| \sum_{l=0}^{r} W_l^{\bd_{2^n}} \right|\ge
   \eps c_{2^n}\right)$ we apply~\cite{merl}[Inequality (A.42)], which we now recall.

\textbf{\cite{merl}[Inequality (A.42)].} Let $X_1,\dots,X_n$ be random variables with $S_n=\sum_{i=1}^n X_i$ and $S_n^*=\max_{k=1,\dots,n} |S_k|$.
Introduce the sigma algebras $\mathcal F_u=\sigma(X_1,\ldots,X_u)$ if $u> 0$
    and $\mathcal F_u=\{\emptyset, \Omega\}$ if $u\le 0$, and let $\mathbb{E}_q$ denote conditional expectation with respect to $\mathcal{F}_q$.
    Let $q,M,\lambda$ be positive constants with $qM\le \lambda$, and let $\phi$ be a nondecreasing, even, nonnegative, convex function.
    Introduce furthermore the notations $X_i'=X_i\cdot \mathbf{1}_{|X_i|\le M}$ and $X_i''=X_i-X_i'$ for $i=1,\dots, n$. Then the following inequality
    can be derived from the Doob maximal inequality, see \cite{merl}[pp. 411--412]:

\begin{equation}\label{eq:MerlRioA42}
     \mathbb{P}(S_n^*\ge 5\lambda)\le \frac{\mathbb{E}(\phi(S_n))}{\phi(\lambda)} + \lambda^{-1} \sum_{i=1}^n \mathbb{E} |X_i''| +
     \lambda^{-1} \sum_{i=1}^n \|\mathbb{E}_{i-q} (X_i') \|_1.
     \end{equation}

We apply \eqref{eq:MerlRioA42}  to
   the random variables
   \begin{align*}
    &X_{l+1}=W_l^{\bd_{2^n}},\text{ which are  bounded by }M(n)=\bd_{2^n}, \text{ with }
    0\le l\le  |I_{n,j}|, \text{ hence}\\
    &X_l'=(X_l\wedge M)=X_l\text{ and thus } X_l''=X_l- X_l'=0.
    \end{align*}
  Also, $\mathcal F_u=\sigma(W_0^{\bd_{2^n}},\ldots,W_u^{\bd_{2^n}})$ if $u\le 0$
    and $\mathcal F_u=\{\emptyset, \Omega\}$ if $u<0$.
To proceed we exploit (HL3). We recall that $q(n)=\lfloor\delta_0\frac{c_{2^n}}{\bd_{2^n}}\rfloor$ for $\delta_0<1$,
and we take $\delta_0=\eps$.
    Set $\lambda(n)=\eps c_{2^n}$ and note that $q(n)M(n)\le \lambda(n)$.
    So, the prelimnary condition for applying the intended inequality is satisfied.
    Taking $\phi(x)=x^4$ (which is a convex, non-negative and even function),
    \eqref{eq:MerlRioA42} reads as

\begin{align}\label{ineqmr}
 \mu\left(\max_{r<|I_{n,j}|}\left| \sum_{l=0}^{r} W_l^{\bd_{2^n}} \right|\ge
   \eps c_{2^n}\right) &\le \frac{\int \left|\sum_{l\in I_{n,j}} W_{l}^{\bd_{2^n}}\right|^4\, d\mu}{(\eps c_{2^n})^4}
   +\frac{\sum_{l=0}^{|I_{n,j}|} \E_\mu(|X_l''|)}{\eps c_{2^n}}\\
  \nonumber&+
\frac{\sum_{l=0}^{|I_{n,j}|}\|\E_\mu(X_l|\mathcal F_{l-q(n)})\|_{L^1(\mu)}}
{\eps c_{2^n}}.
\end{align}
 Clearly, $\E_\mu(|X_l''|)=0$, so the middle term on RHS of~\eqref{ineqmr} vanishes.
For the third term on the RHS of~\eqref{ineqmr} we will use (HL3). Using the notation of \eqref{eq:sigma_algebras} and (HL3) we have
\begin{align*}
\|\E_\mu(X_l|\mathcal F_{l-q(n)})\|_{L^1(\mu)} &=\|\E_\mu(W_{l}^{\bd_{2^n}}\Big|\mathcal{F}_{0,l-q(n)}(W^{\bd_{2^n}}))\|_{L^1(\mu)}\\
&\ll \bd_{2^n} \gamma^{q(n)-C_3\log(\bd_{2^n})}  + (\bd_{2^n})^{-10}.
\end{align*}

Recalling $q(n)=\lfloor\eps\frac{c_{2^n}}{\bd_{2^n}}\rfloor$
\begin{align*}
 \frac{\sum_{l=0}^{|I_{n,j}|}\|\E_\mu(X_l|\mathcal F_{l-q(n)})\|_{L^1(\mu)}}
{\eps c_{2^n}}\ll \frac{|I_{n,j}|}{c_{2^n}} \cdot\left(\bd_{2^n}\gamma ^{q(n)-C_3\log(\bd_{2^n})}+ \bd_{2^n}^{-10} \right)
\end{align*}
for $\gamma<1$. Recalling $\bd_{2^n}$ as in~\eqref{eq:dn}, we have $\frac{c_{2^n}}{\bd_{2^n}}\gg \frac{2^{n/2}}{2^{n(1-\varsigma)/2}}=2^{n\varsigma/2}$. In particular,
$q(n)\gg 2^{n\varsigma/2}$. So,
\begin{align}\label{3t}
 \frac{\sum_{l=0}^{|I_{n,j}|}\|\E_\mu(X_l|\mathcal F_{l-q(n)})\|_{L^1(\mu)}}
{\eps c_{2^n}}\ll |I_{n,j}| \cdot\left( 2^{-n\varsigma/2} \gamma ^{2^{n\varsigma/2}} + 2^{-4n} \right) \ll |I_{n,j}|\cdot 2^{-4n}.
\end{align}

Regarding the first term in~\eqref{ineqmr}, we just use (H3)(ii).
 By (H3)(ii) with $R=\bd_{2^n}$ and $n$ replaced by $|I_{n,j}|$ (note that $|I_{n,j}|\asymp 2^{(1-\beta)n}=
o((\bd_{2^n})^{2-r_1})$ for some $r_1>0$, if $\varsigma$ is sufficiently small),
\[
\int \left|\sum_{l=0}^{|I_{n,j}|} W_{l}^{\bd_{2^n}}\right|^4\, d\mu \, \ll \, |I_{n,j}| (\bd_{2^n})^2\ll \, |I_{n,j}| (c_{2^n})^2.
\]
So,
\begin{align}\label{t2}
 \frac{\int \left|\sum_{l\in I_{n,j}} W_{l}^{\bd_{2^n}}\right|^4\, d\mu}{(\eps c_{2^n})^4}\ll \frac{|I_{n,j}|}{(c_{2^n})^2}.
\end{align}
Using~\eqref{3t},~\eqref{t2} and~\eqref{ineqmr},
\begin{align*}
 \mu\left(\max_{r<|I_{n,j}|}\left| \sum_{l=0}^{r} W_l^{\bd_{2^n}} \right|\ge
   \eps c_{2^n}\right)\ll |I_{n,j}|\cdot 2^{-4n}
   +\frac{|I_{n,j}|}{(c_{2^n})^2}.
\end{align*}
Summing on $j=0,\dots, F(n)-1$ and using that
$\sum_{j=0}^{F(n)-1}|I_{n,j}|\asymp 2^n$, we obtain
\[
 \sum_{j=0}^{F(n)-1} \mu\left(\max_{r<|I_{n,j}|}\left|\sum_{l=i_{n,j}}^{i_{n,j}+r}W_l^{\bd_{2^n}}\right|\ge \eps c_{2^n}\right)
\ll 2^{-3n} + \frac{2^n }{2^n n \ell_1(2^n)},
 \]
which is summable in $n$, as required.

\end{proofof}

\subsection{Proof of Lemma~\ref{lemma:gaps}}
\label{sec:gaps}

Since we are looking at the truncated process for which the second moment is finite,
 it makes sense to apply the Gal-Koksma strong law of large numbers~\cite[Proposition 3.7]{Gou-asip}. As such,
we can proceed as~\cite[Section 5.4]{Gou-asip}.
We recall the main steps in our set up, which are slightly
different because the effects of the truncation have to be dealt with.

Our key Lemma, analogous to~\cite[Proof of Lemma 5.9]{Gou-asip}, is
\begin{lemma}\label{lemma:ggaps}
For any $\eps>0$, there exists $C>0$ so that for any interval $J\subset\N$,
\[
 \E_\mu\left(\sum_{l\in J\cap\mathcal J} W_l^{c_{\hl}}\right)^2\le C\, |J\cap\mathcal J|^{1+\eps}.
\]
\end{lemma}
This allows to complete in the same manner as in~\cite[Section 5.4]{Gou-asip},

\begin{pfof}{~Lemma~\ref{lemma:gaps}} The Gal-Koksma strong law of large numbers~\cite[Propostion 3.7]{Gou-asip} says that if $(X_l)_{l\ge 0}$ are centered
random variables and if
$\E\left|\sum_{l=m}^{m+n-1} X_l\right|^2 \le C n^r$ for some $r\ge 1$ and some $C>0$, then for any $\eps_2>0$, $\frac{\sum_{l=1}^N X_l }{N^{r/2+\eps_2}}\to 0$, almost surely, as $N\to\infty$.

We can apply Gal-Koksma to $W_l^{c_{\hl}}$; this is possible because by assumption (H3)(i), $W_l^{c_{\hl}}$ are centered.

By  Lemma~\ref{lemma:ggaps} for any $\eps>0$,
\(
 \E_\mu\left(\sum_{l\in J\cap\mathcal J}W_l^{c_{\hl}}\right)^2\le C |J\cap\mathcal J|^{1+\eps}.
\)
Taking $k\in [2^n, 2^{n+1})$ and recalling~\eqref{gaps},
\[
 |J\cap\mathcal J|\le \sum_{m=1}^n\sum_{j=0}^{F(m)-1}|J_{m,j}|\ll \sum_{m=1}^n m 2^{\eps_1 m+\beta m}\ll n 2^{\eps_1 n+\beta n}\ll k^{\beta+3\eps_1/2}.
\]
It follows from Gal-Koksma that
\(
 \sum_{l<k, l\in \mathcal J} W_l^{c_{\hl}}=
 o\left(k^{\beta/2+2\eps+\eps_1}\right),
\)
as desired.~\end{pfof}

It remains to provide

\begin{pfof}{~Lemma~\ref{lemma:ggaps}}
 As already mentioned, this proof is analogous to~\cite[Proof of Lemma 5.9]{Gou-asip}. The  only aspect which is different in our set up is that for estimating second moments of (truncated) sums we need to use (H3)(iii) instead of
 ~\cite[Proposition 4.1]{Gou-asip} and the current (H2) (instead of \cite[Assumption (H)]{Gou-asip}).

 To start, let us recall that $L(\cdot)=\log(\max(\cdot,e))$ and introduce the notation
\begin{equation}
\label{eq:whatisN}
N=\max_{l\in J} L(l).
\end{equation}
Via (H3), second moment moment bounds of sums of truncated variables will include factors of the type  $\ell^*(c_l)$ for various $l\in J$. All these contributions can be estimated
using
\[
\ell^*(c_{2^n})\asymp L(2^n) \ll N.
\]

Now use, as in the second displayed equation inside~\cite[Proof of Lemma 5.9]{Gou-asip}, the convexity inequality
 \begin{align*}
  \left(\sum_{l\in J\cap\mathcal J} W_l^{c_{\hl}}\right)^2
  \le 3\left(\sum_{l\in J_0} W_l^{c_{\hl}}\right)^2+3 \left(\sum_{l\in J_1} W_l^{c_{\hl}}\right)^2 +3\left(\sum_{l\in J_0} W_l^{c_{\hl}}\right)^2,
 \end{align*}
where $J_0, J_2$ are the first and last intervals in $J\cap\mathcal J$
and where $J_1$ is the remaining part.

Note that $ \E_\mu\left(\sum_{l\in J_0} W_l^{c_{\hl}}\right)^2\ll  \E_\mu\left(\sum_{l\in J_0} W_l^{\max_{l\in J_0}c_l}\right)^2$ and $ \E_\mu\left(\sum_{l\in J_2} W_l^{c_{\hl}}\right)^2\ll  \E_\mu\left(\sum_{l\in J_2} W_l^{\max_{l\in J_2}c_l}\right)^2$. This together with (H3) gives
\begin{align*}
  \E_\mu\left(\sum_{l\in J_0} W_l^{c_{\hl}}\right)^2\ll |J_0|\cdot\ell^*(\max_{l\in J_0}c_l),\qquad
  \E_\mu\left(\sum_{l\in J_2} W_l^{c_{\hl}}\right)^2\ll |J_2|\cdot\ell^*(\max_{l\in J_2}c_l).
\end{align*}
Under (HL0), for all $m$ large enough, $\ell^*(m)\approx L(m)$. Recalling the assumption on $c_n$ in (HL0),
\begin{align}\label{eq:j0j2}
 \E_\mu\left(\sum_{l\in J_0} W_l^{c_{\hl}}\right)^2\ll |J_0|\cdot N,\qquad
  \E_\mu\left(\sum_{l\in J_2} W_l^{c_{\hl}}\right)^2\ll |J_2|\cdot N.\end{align}

We claim that for any $\eps>0$,
\begin{align}\label{eq:j1}
  \E_\mu\left(\sum_{l\in J_1} W_l^{c_{\hl}}\right)^2\ll |J_1|^{1+\eps}.
\end{align}
This together with~\eqref{eq:j0j2} gives the conclusion.

If $J_1$ is empty, there is nothing to do. If $J_1$ is not empty, the proof of the claim~\eqref{eq:j1} follows the argument
of~\cite[Proof of Equation (5.18)]{Gou-asip} with again, the only difference that we
have to use
(H3) instead of~\cite[Proposition 4.1]{Gou-asip} and we need the current (H2).

Recall that as before~\eqref{gaps}, an interval has rank $r$ if its length is $2^{\lfloor \eps_1 n\rfloor +r}$.
Let $\mathcal J^{(n,r)}$ denote the union of intervals $J_{n,j}$ of rank $0\le r\le\rfloor n\beta\rfloor$.
Recall, furthermore, \eqref{eq:whatisN}. The number of sets $\mathcal J^{(n,r)}$ intersecting $J_1$ is at most $\sum_{n=0}^N (\rfloor n\beta\rfloor+1)\ll N^2$.
Proceeding as in the argument used in obtaining~\cite[Equation (5.19)]{Gou-asip}, we can rely on the convexity inequality and get
\begin{align}\label{eq:j11}
 \left(\sum_{l\in J_1} W_l^{c_{\hl}}\right)^2\ll  N^2\sum_{n,r}\left(\sum_{l\in J_1\cap \mathcal J^{(n,r)}} W_l^{c_{\hl}}\right)^2.
\end{align}

 Let us fix some relevant $(n,r)$ and enumerate the intervals in $\mathcal J^{(n,r)}$ as $K_1,\ldots, K_t$
 for $t=2^{\rfloor n\beta\rfloor -1-r}$, $0\le r\le\rfloor n\beta\rfloor$.
 Let $T_s=\sum_{l\in K_s} W_l^{c_{\hl}}$.
To complete the proof of Lemma~\ref{lemma:ggaps}, the key bound, analogous to~\cite[Equation (5.20)]{Gou-asip}, is that for any subset $S\subset\{1,\dots,t\}$
\begin{align}\label{eq520}
  \E_\mu\left(\sum_{s\in S} T_s\right)^2\ll \left( \sum_{s\in S}\E_\mu(T_s^2)+|S|\right)\cdot N.
 \end{align}
 \eqref{eq520} can be proved using (H2) and (H3), see below, but let us first see how to obtain \eqref{eq:j1} from \eqref{eq520}.
By (H3), $\E_\mu(T_s^2)\ll \left(\sum_{l\in K_s} W_l^{\max_{l\in K_s}c_l}\right)^2\ll
|K_s|\cdot N$.
This type of bound holds for any set $K$ which is a union of intervals in $\mathcal J^{(n,r)}$.
Taking $K=J_1\cap \mathcal J^{(n,r)}$ and combining with~\eqref{eq:j11},
\begin{align*}
 \left(\sum_{l\in J_1} W_l^{c_{\hl}}\right)^2\ll  N^2\sum_{n,r} |J_1\cap \mathcal J^{(n,r)}|\cdot N
\end{align*}
and the claim~\eqref{eq:j1} follows as $N\ll |J_1]^{\eps}$ for any $\eps>0$, whenever $J_1$ is nonempty (note that if $J_1$ is nonempty, it has to contain at least one interval of length $\asymp 2^{\beta N}$).

It remains to justify~\eqref{eq520} proceeding as in the proof of~\cite[Equation (5.20)]{Gou-asip}.
Let $(U_1,\ldots, U_t)$ be independent random variables such that $U_s$ is distributed like $T_s$. We recall that given $\eps_0$ as in (H2),~\cite[Proposition 3.8]{Gou-asip} ensures that there exists
 a symmetric random variable $V$ in $\R^d$, $V\in L^q$, $q\ge 2$, whose characteristic function is supported in $\{|t|\le\eps_0\}$. Let $V_1,\ldots, V_t, V_1',\ldots, V_t'$ be independent random variables distributed like $V$ and define $\tilde T_s= T_s+V_s$,
 $\tilde U_s=U_s+V_s'$. Using intervals of rank greater than $k$ as gaps and using the procedure recalled in justifying~\eqref{eq:g3} above we obtain (in the same manner
 as in the proof of~\cite[Equation (5.21)]{Gou-asip}),
 \begin{align*}
  \pi\left((\tilde T_1,\ldots,\tilde T_t),(\tilde U_1,\ldots,\tilde U_t)\right)\ll e^{2^{\delta n}},
 \end{align*}
for some $\delta>0$. Now we may proceed as in \cite{Gou-asip}: by Theorem~\ref{thm-StrD} there exists a coupling between the $(\tilde T_j)$ and the $(\tilde U_j)$ $(j=1,\dots, t)$ such that,
outside a set $O$ of measure $\ll e^{-2^{\delta n}}$, we have $|\tilde T_j-\tilde U_j|\ll e^{-2^{\delta n}}$ for any  $j=1,\dots, t$. Then, for any subset $S\subset\{1,\dots,t\}$ we have
\begin{equation} \label{eq:tildeTs}
\left\| \sum_{s\in S} \tilde{T}_s \right\|_{L^2} \le \left\| \mathbf{1}_O \cdot \sum_{s\in S} \tilde{T}_s \right\|_{L^2} +
\left\| \mathbf{1}_{O^c} \cdot \sum_{s\in S} (\tilde{T}_s-\tilde{U}_s) \right\|_{L^2}
+ \left\| \sum_{s\in S} \tilde{U}_s \right\|_{L^2}.
\end{equation}

To estimate the first term in \eqref{eq:tildeTs}, \cite{Gou-asip} uses an $L^p$ bound for a relevant $p>2$, here we proceed slightly differently; we rely on (H3i) and get
\begin{align*}
\left\| \mathbf{1}_O \cdot \sum_{s\in S} \tilde{T}_s \right\|_{L^2} &\le    \left\| \mathbf{1}_O \right\|_{L^4} \cdot \left\| \sum_{s\in S} \tilde{T}_s \right\|_{L^4} \\
&\ll e^{-2^{\delta n}/4} \cdot \sum_{s\in S} \left\| \tilde{T}_s \right\|_{L^4} \ll e^{-2^{\delta n}/4} \cdot |S| \cdot 2^{n/4} \cdot c_{2^n}^{1/2} \\
&\ll  e^{-2^{\delta n}/4} \cdot  2^{2n} \le C
\end{align*}
 for some $C>0$. From this point on, the argument can be completed literally as in \cite{Gou-asip}; the second term in \eqref{eq:tildeTs} can be bounded by some $C>0$, and thus
 \begin{align*}
 \left\| \sum_{s\in S} T_s \right\|_{L^2} &\le \left\| \sum_{s\in S} \tilde{T}_s \right\|_{L^2} + \left\| \sum_{s\in S} V_s \right\|_{L^2}\\
 &\le C + \left\| \sum_{s\in S} U_s \right\|_{L^2} +\left\| \sum_{s\in S} V_s \right\|_{L^2} + \left\| \sum_{s\in S} V'_s \right\|_{L^2}\\
 &\le C + (\sum_{s\in S} \mathbb{E}_{\mu} (T_s^2))^{1/2} +C|S|^{1/2}.
 \end{align*}
\end{pfof}

\section{Verification of assumptions (H3)(ii), (HL1), (HL2) and (HL3) for the Lorentz gas}
\label{sec:fm}

\subsection{Verification of (H3)(ii): bound on the fourth moment}

Let $\cS=\sum_{i=0}^{n-1} W^R\circ T^i$, that is, the number of terms is $n$ and the truncation level is $R$. Throughout, we assume $n=o(R^{2-r_1})$ and $R=O(n^{r_2})$, where $r_1>0$ and $r_2>0$ are some
fixed exponents - this implies also $\log R \asymp \log n$.
Our aim is to bound $\E (\cS^4)$. For brevity let $Y_i=W^R\circ T^i$ and $Y=Y_0$. We also assume $\E Y=0$ which holds in the Lorentz gas by symmetry.
We have
\[
\E \cS^4 \ll \sum_{0\le i\le j\le k \le l\le n-1 } \E(Y_i Y_j Y_k Y_l)
\]
We distinguish several cases depending on the indices $i,j,k,l$. For some large constant $C$, let
\begin{description}
\item Case 1. Either $|i-j|\ge C\log n$ or $|k-l|\ge C\log n$. By decay of correlations (see eg.~\cite[Lemma 4.1]{BCDcusp}) the contribution of such terms is
\[\ll n^4 \cdot R^4 \theta^{C\log n}= O(1)\]
here $n^4$ is the number of terms and for $C$ sufficiently large the decay of correlations can be made $\theta^{C\log n}\ll n^{-10r_2+10}$ so that it kills all polynomially growing factors.
\item Case 2. Both $|i-j|\le C\log n$ and $|k-l|\le C\log n$ but $|j-k|\ge C\log n$, we may again apply decay of correlations but now the expectations $\E(Y_i Y_j)$ (and $\E(Y_k Y_l)$) are nonzero.
Using the Cauchy-Schwartz inequality, any such term contributes by $(\E (Y^2))^2\ll (\log R)^2$ (correlations give an $O(1)$ correction) hence the total contribution is
\[
n^2 (\log R)^2 (\log n)^2 \ll n^2 (\log R)^4\ll n R^{2-r_1} (\log R)^4\ll n R^2.
\]

\item Case 3. $|i-j|\le C\log n$, $|k-l|\le C\log n$ and $|j-k|\le C\log n$.
As further bounds will rely on it, first let us estimate $\E(Y_i^2 Y_j^2)=\E(Y^2 Y_{j-i}^2)$ . We again distinguish two cases.
\begin{description}
\item Case a, $i=j$. We only have $\E Y^4\ll \sum_{m=1}^R m^4 m^{-3} \ll R^2$.
\item Case b, $j>i$. Let
\begin{equation}\label{eq:Mmdef}
M_m=\{\omega\in M | \lfloor |v(\omega)| \rfloor=m\}
\end{equation}
for some $m\ge 1$. Then $\mu(M_m)\ll m^{-3}$ and (see \cite[Formula (9.6)]{ChDo09})
\[
\mu(M_m \cap T^{-(j-i)} M_{m'})\ll m^{-9/4} (m')^{-2}.
\]
This gives
\[
\E(Y^2 Y_{j-i}^2)\ll \sum_{m=1}^R \sum_{m'=1}^R m^2 (m')^2 m^{-9/4} (m')^{-2}\ll R^{3/4} R  =R^{7/4}.
\]
\end{description}

Now let us get back to the investigation of Case 3 above. Let us consider first the terms when $i=j$ and $k=l$. If, furthermore, $j=k$ then we get a single term when all indices coincide contributing with
\[
\E (Y^4) \ll R^2,
\]
see above. Otherwise, if $i=j<k=l$, (yet $k-j<C\log n$) we may apply Case b above to obtain
\[
\E (Y_j^2Y_k^2) \ll R^{7/4}
\]
and as there are $\ll n \log n$ such terms (first fix $(i=)j$, then can pick $k(=l)$ in $\ll \log n$ ways) their total contribution is
\[
\ll R^{7/4} n \log n \ll n R^2.
\]
Now if either $j>i$ or $l>k$ (yet the conditions of Case 3 remain valid), first we apply a Cauchy-Schwartz to such terms and get
\[
\ll (\E(Y_i^2 Y_j^2))^{1/2} \cdot (\E(Y_k^2 Y_l^2))^{1/2}.
\]
As either $j>i$ or $l>k$, to at least one of the factors Case b applies and we get $\ll R^{7/8} R= R^{15/8}$. the total contribution of such terms is
\[
n R^{15/8} (\log R)^3 \ll n R^2.
\]
\end{description}

\begin{rmk}
\label{rmk:whysmallmnoproblem}
We note that the condition $R=o(n^{r_2})$ is for mere simplicity exposition, to ensure $\log n \asymp \log R$. If the number of terms $n$ is less, so that $R=o(n^{r_2})$ does not apply,
we could still get the same type of bound with splitting the cases on conditions of indices in terms of $C \log R$ instead of $C\log n$.
\end{rmk}

\subsection{Verification of (HL1)}

Let us recall the notations: $\tW= v\cdot \mathbf{1}_{d_{2^n}\le|v|\le c_{2^n}}$,  $\tW_{l}=\tW\circ T^{l+r}$ for $l\ge 1$ and $r$ fixed, while, for some
$N\le (c_{2^n})^2$, $\tcS_N=\sum_{l=0}^{N} \tW_l$.

By exponential decay of correlations, see \cite[Proposition 9.1]{ChDo09}, the contribution of cross terms is of $O(N)$, and thus (see also \cite[Proof of Lemma 11.1]{ChDo09})
\begin{align}
\label{eq;secmtw}
 \E_\mu\left((\tcS_N)^2\right)\ll  N \cdot L\left( \frac{c_{2^n}}{d_{2^n}}\right)\ll  N\cdot LL(n),
\end{align}
as required.

\subsection{Verification of (HL2)}
This time we work with the lower truncation level $\bd_{2^n}$ from \eqref{eq:dn}, and the notations are $\bW= v\cdot \mathbf{1}_{\bd_{2^n}\le|v|\le c_{2^n}}$,  $\bW_{l}=\bW\circ T^{l+r}$ for $l\ge 1$ and $r$ fixed,
$K$ is fixed so that $K< (\bd_{2^n})^2$. Furthermore, for $k=1,\dots, K$ and $\eps>0$,
\[
\cS_k=\sum_{l=0}^{k} \bW_l;\qquad B_{k,\eps}=\{x\in M \,:\, |\cS_k|\ge \eps c_{2^n}\}
\]
and
\(
N_{k,\eps}=B_{k,\eps}\setminus (\cup_{k'=0}^{k-1} B_{k',\eps}).
\)

We need to show that
 \(
\sum_{k=1}^K \mu(N_{k,\eps} \setminus B_{K,\eps/2})\le C_2 \frac{K}{(c_{2^n})^2}.
\)
A similar type of bound is obtained in \cite[pp.696--697]{ChDo09} and \cite[pp.~506--508]{BCDcusp}, our derivation is analogous, nevertheless, we use the fourth moment bound of (H3)(ii)
instead of the second moment bound. To start,
note that the points $x\in N_{k,\eps} \setminus B_{K,\eps/2}$ simultaneously satisfy three conditions:
\begin{equation}\label{eq:max3cond}
|\cS_k|\ge \eps \cdot c_{2^n};\qquad |\cS_K-\cS_k|\ge \frac{\eps}{2} \cdot c_{2^n}; \qquad  \bW_k\ne 0.
\end{equation}
Let us consider the third condition: $\bW(x) \ne 0$ implies that $x\in M_m$ for some $m$ with $\bd_{2^n}\le  m \le c_{2^n}$
(recall \eqref{eq:Mmdef} for the definition of $M_m$). Now we may use \cite[Lemma 5.1]{ChZh08} (see also \cite[Lemma 16]{SV07}):
there exist constants $p,q>0$ such that for any $C>0$ there exists a subset $\tM_m\subset M_m$ such that
\begin{align*}
&\mu(M_m\setminus \tM_m) \ll m^{-p} \mu(M_m); \quad \text{and} \\
& \quad x\in \tM_m \ \Rightarrow\ T^i x\not\in M_r \ \text{for} \ r>m^{1-q} \ \text{whenever}\  i=1,\dots C \cdot L(m).
\end{align*}
Since
\[
\mu\left( \cup_{m\ge \bd_{2^n}} (M_m\setminus \tM_m)\right) \ll \sum_{m\ge \bd_{2^n}} m^{-(3+p)} \ll (\bd_{2^n})^{-(2+p)}=o(c_{2^n}^{-2}),
\]
we may assume that whenever $\bW_k\ne 0$, we have $T^{(k+i_{n,j})}x\in  \tM_m$ for some $m$ with $\bd_{2^n}\le  m \le c_{2^n}$. However, for $\varsigma$ sufficiently small, this implies
that $\bW_{k'}= 0$ for $k'=(k+1),\dots, k+ C n$ (recall $L(\bd_{2^n}) \asymp n$). In terms
of the three conditions of \eqref{eq:max3cond} this implies that  the contribution of $\mu(N_{k,\eps} \setminus B_{K,\eps/2})$ can be disregarded unless $K-k> C n$.
Furthermore, the values $\bW_k\ne 0$ and $\bW_{k'}\ne 0$  for $k'>k+Cn$ are strongly decorrelated (by exponential decay of correlations, see eg.~\cite[Lemma 4.1]{BCDcusp}) and are of order
$\gamma^n$ for some $\gamma<1$ that can be made arbitrary small by choosing $C$ large.
To conclude, let
\[
F_k= |\cS_k| \cdot \mathbf{1}_{N_{k,\eps}};\qquad G_k= |\cS_K-\cS_k|.
\]
Then, using (H3)(ii),
\begin{align*}
\mu(N_{k,\eps} \setminus B_{K,\eps/2}) \le & \mu\left( F_k\ge \eps c_{2^n} \ \text{and} \ G_k\ge \frac{\eps}{2} c_{2^n}\right) \\
\le & \, \mu\left( F_k\cdot G_k \ge \frac{\eps^2}{2}  (c_{2^n})^2 \right) \\
\ll & \, \frac{\mu(F_k^4 G_k^4)}{c_{2^n}^{8}} \ll \left[ \mu(F_k^4) \mu(G_k^4) +\gamma^{n} \right] / c_{2^n}^8
\end{align*}
and the contribution from $\gamma^{n}$ is under control. As in the proof of \eqref{eq:nomaxd_n_c_n}, we may apply the fourth moment bound (H3)(ii) (cf.~also Remark~\ref{rmk:whysmallmnoproblem}) to obtain
\[
\mu(G_k^4) \ll (K-k)\cdot (c_{2^n})^2\ll K \cdot (c_{2^n})^2.
\]
To bound $\mu(F_k^4)$, note that for $x\in N_{k,\eps}$
\[
|\cS_k(x)|\le |\cS_{k-1}(x)|+ |\bW_k(x)| \ll \eps c_{2^n} + c_{2^n} \ll c_{2^n}
\]
thus
\[
\mu(F_k^4)\ll (c_{2^n})^4 \cdot \mu(N_{k,\eps}) \ll \frac{(c_{2^n})^4}{ (\bd_{2^n})^{2}}.
\]
Putting all these estimates together and recalling that $K\le \bd_{2^n}^2$ by assumption,
\[
\sum_{k=1}^K \mu(N_{k,\eps} \setminus B_{K,\eps/2}) \ll K \cdot \frac{K \cdot (c_{2^n})^6}{(c_{2^n})^8 (\bd_{2^n})^2} \ll \frac{K}{(c_{2^n})^2},
\]
as required.

\subsection{Verification of (HL3) \label{ss:HL3verify}}

Recall the notation of \eqref{eq:sigma_algebras}. We will prove a statement slightly more general than (HL3):  there exist some $\gamma<1$ and $C>0$ and $C_3>0$ such that, for any $R\ge 0$ (truncation level) and $q\ge 1$ (correlation gap) satisfying $q\ge \log (R)$, and any $k\ge 0$ we have
\begin{equation}\label{eq:HL3i}
 \left\|\E_\mu\left(W_{k+q}^{R}\Big|\cF_{0,k}(W^{R})\right)\right\|_{L^1(\mu)}\le C (R \gamma^{q-C_3\log(R)} + R^{-10}).
\end{equation}

Our argument will use some tools from the theory of hyperbolic billiards: unstable curves, unstable manifolds, standard families, Z functions, the growth lemma, and the phenomena of equidistribution.

Unstable \textit{curves} are curves with tangents in the unstable cone. Through $\mu$-almost every point $x\in M$ there exists a unique, special unstable curve, $\eta^u(x)(=\eta^u)$,
the \textit{unstable manifold} of $x$, characterized by the property that for $y\in \eta^u(x)$ we have $d(T^{-n}x,T^{-n}y)\to 0$ exponentially as $n\to\infty$. We will denote by $\mathcal{U}$
the sigma-algebra generated by unstable manifolds.

Standard families $\mathcal{G}$ are appropriate probability measures on $M$ that can be
desintegrated into standard pairs, that is, probability measures smoothly supported on unstable curves.
The Z function of a standard family, denoted as $\mathcal{Z}_{\mathcal{G}}$, is a quantity expressing
the average inverse length of the unstable curves included in this family. The Growth Lemma describes the evolution of the Z function under the dynamics: there exist $C_1>0, C_2>0$ and
$\theta<1$ such that, for any standard family $\mathcal G$ we have
\begin{equation}
\label{eq:growth}
\mathcal{Z}_{T^n_*\mathcal{G}} \le C_1 \theta^n \cdot \mathcal{Z}_{\mathcal{G}}+C_2,
\end{equation}
where $T^n_*\mathcal{G}$ is the standard family obtained by pushing forward $\mathcal{G}$ by $T^n$.
A standard family is proper if its Z function is less than some uniform $C>0$. In uniformly hyperbolic billiards, such as the Lorentz gas map,
proper standard families equidistribute at an exponential rate: there exists some $\gamma<1$ such that the $T^n$-push-forward of a proper standard family is $\gamma^n$-close to the invariant
Liouville measure $\mu$.
We refer to \cite[Chapter 7]{ChM}, \cite[Section 6]{ChDo09} and \cite[Section C.1]{BBT} for further details.

Now, let us recall that we want to prove \eqref{eq:HL3i}. As the billiard map is an invertible dynamical system, we have a bi-infinite stationary sequence of random variables
$\dots,W_{-1}^R,W_0^R, W_1^R,\dots$ and by stationarity, the LHS of \eqref{eq:HL3i} can be rewritten as
\[
 \left\|\E_\mu\left(W_q^{R}\Big|\cF_{-k,0}(W^{R})\right)\right\|_{L^1(\mu)}.
\]
Recall that $\mathcal{U}$ is the sigma algebra generated by unstable manifolds. Note that whenever $y\in\eta^u(x)$, we have $\kappa(T^{-k}x)=\kappa(T^{-k}y)$, and thus $W^{R}(T^{-k}x)=W^{R}(T^{-k}y)$ for any $k\ge 0$.
This implies that for any $k\ge 0$, the sigma algebra $\cF_{-k,0}(W^{R})$ is coarser than $\mathcal{U}$, and thus
\[
\left\|\E_\mu\left(W_q^{R}\Big|\cF_{-k,0}(W^{R})\right)\right\|_{L^1(\mu)}\le \left\|\E_\mu\left(W_q^{R}\Big|\,\mathcal{U}\,\right)\right\|_{L^1(\mu)}.
\]
It remains to prove that
\begin{equation}\label{eq:HL3ii}
\left\|\E_\mu\left(W_q^{R}\Big|\,\mathcal{U}\,\right)\right\|_{L^1(\mu)}\le  C (R \gamma^{q-C_3\log(R)} + R^{-10})
\end{equation}
for some $C>0, C_3>0$ and $\gamma<1$. To proceed, denote the unstable manifolds included in $\mathcal{U}$ by $\eta_i$, $i\in I$, where the index set $I$ is uncountable.
When conditioning $\mu$ on $\eta_i$, we obtain a probability measure $\mu_i$, and thus a standard pair $\mathcal{G}_i=(\eta_i,\mu_i)$. For each $i\in I$ and any $k\ge 0$, the
push-forward $\mathcal{G}_i^k=T^k_*\mathcal{G}_i$ is a standard family. Then, the conditional expectation  $\E_\mu\left(W_q^{R}\Big|\,\mathcal{U}\,\right)$ is a random variable that takes a
constant value on each $\eta_i$, in particular the average of $W_0^R$ with respect to $\mathcal{G}_i^k$. To estimate these averages (and thus to get the bound on the $L^1$-norm as in \eqref{eq:HL3ii}),
we classify the unstable manifolds $\eta_i$ included in $\mathcal{U}$ according to their length. We refer to the following tail estimate on the length of unstable manifolds from \cite[Theorem 4.75]{ChM}:
there exists some $C>0$ such that, for any $\varepsilon>0$:
\begin{equation}\label{eq:UnstableTail}
\mu(x\in M \,|\,|\eta^u(x)|<\varepsilon) \le C\varepsilon,
\end{equation}
where $|\eta^u(x)|$ denotes the length of the unstable manifold of $\eta^u(x)$.
\begin{description}
\item \textit{Contribution of short unstable manifolds.} According to \eqref{eq:UnstableTail}, the $\mu$-measure of points with unstable manifolds shorter than $R^{-11}$ is less than $C\cdot R^{-11}$. On the other hand, by truncation, the function $W_0^R$ (and thus its conditional expectation) has absolute value uniformly bounded by $R$. Thus such short curves make a contribution to the $L^1$-norm of the conditional expectation bounded by $R\cdot C R^{-11}\ll R^{-10}$, giving the second term on the RHS of \eqref{eq:HL3ii}.
\item \textit{Contribution of long unstable manifolds.} Now consider the contribution of unstable manifolds of length $|\eta_i|\ge R^{-11}$. By \eqref{eq:growth},  there exists some $C_3>0$ such that for any such $\eta_i$ and any $k\ge C_3\log R$,  the $T^k$-push-forward $\mathcal{G}_i^k$ is a proper standard family. Hence its further iterates equidistribute at an exponential rate. Consequently, for $q\ge C_3\log R$, averages with respect to $\mathcal{G}_i^q$ are $\gamma^{q-C_3\log R}$-close to averages with respect to $\mu$. Now by (H3)(i), we have $\mathbb{E}_\mu (W_0^R)=0$, thus the average of $W_0^R$ with respect to $\mathcal{G}_i^q$ is uniformly bounded by $\ll R \cdot \gamma^{q-C_3\log R}$,  for long unstable manifolds $\eta_i$. This gives the first term on the RHS of \eqref{eq:HL3ii}, and thus completes the verification of (HL3).
\end{description}

\appendix

\section{The  sequence $c_n$ defined in~\eqref{eq:cnlg} satisfies (HL0)(iii)}
\label{sec:cn}

The argument in~\cite[p.~1622]{EinLi} shows that given $c_n$ as in~\eqref{eq:cnlg},
$\sum\limits_{k=1}^{\infty} (c_k)^{-2}<\infty$. We need to show that $\sum\limits_{n=1}^\infty \frac{2^n LL(n)}{(c_{2^n})^2}<\infty$.

Proceeding as in the argument in~\cite[p.~1622]{EinLi}, a change of variables shows that
\begin{align*}
 \sum_n \frac{2^n LL(n)}{(c_{2^n})^2}\ll \int_1^\infty \frac{y}{1+e^y \sin^2 y}\, dy.
\end{align*}

We just need to argue that the above integral is finite, which can be seen as follows.
Let $U_k=B_{\frac{1}{(k\pi)^3}}(k\pi)$, $k\ge 1$ and set
$V_k=[(k-\frac 12)\pi, (k+\frac 12)\pi]\setminus U_k$.
Note that
\begin{align*}
 \int_{\frac \pi 2}^\infty \frac{y}{1+e^y \sin^2 y}\, dy=\sum_k\left(\int_{U_k}+\int_{V_k}\right)\frac{y}{1+e^y \sin^2 y}\, dy.
\end{align*}

On $V_k$ we have $\sin^2 y\ge \left(\frac{1}{2k\pi}\right)^6$. Note that $|V_k|=\pi$.
Thus, $\int_{V_k}\frac{y}{1+e^y \sin^2 y}\, dy\le \pi \frac{(2k\pi)^7}{e^{k\pi}}$,
which is summable in $k$.

On $U_k$ we have $\sin^2 y\ge 0$. Note that $|U_k|=\frac{2}{(k\pi)^3}$. Thus,
$\int_{U_k}\frac{y}{1+e^y \sin^2 y}\, dy\le \frac{4k\pi}{(k\pi)^3}$, which is also summable in $k$.\\

{\bf Acknowledgements.} We are grateful to the referees for the careful reading of our manuscript. We are extremely thankful to a referee for pointing
out a problem in the previous proof of Lemma~\ref{lemma:max}
and for suggesting a fix along the lines in~\cite{merl}.

 \end{document}